\documentclass[11pt,twoside]{amsart}

\usepackage{latexsym}
\usepackage{amssymb}
\usepackage{vmargin}
\usepackage{amscd}
\usepackage{stmaryrd}
\usepackage{euscript}
\usepackage{mathrsfs}
\usepackage[all]{xy}
\usepackage{xr}
\usepackage{bbm}

\setmargins{32mm}{20mm}{14.6cm}{22cm}{1cm}{1cm}{1cm}{1cm}

\newcommand\da{\!\downarrow\!}

\newcommand\la{\leftarrow}

\newcommand\id{\mathrm{id}}

\newcommand\ten{\otimes}
\newcommand\hten{\hat{\otimes}}
\newcommand\vareps{\varepsilon}
\newcommand\eps{\epsilon}

\newcommand\CC{\mathrm{C}}

\newcommand\Ru{\mathrm{R_u}}
\newcommand\Th{\mathrm{Th}\,}
\newcommand\CCC{\mathrm{CC}}

\renewcommand\H{\mathrm{H}}
\newcommand\z{\mathrm{Z}}

\newcommand\HH{\mathrm{HH}}

\newcommand\N{\mathbb{N}}
\newcommand\Z{\mathbb{Z}}
\newcommand\Q{\mathbb{Q}}

\newcommand\R{\mathbb{R}}
\newcommand\Cx{\mathbb{C}}

\newcommand\bA{\mathbb{A}}

\newcommand\bC{\mathbb{C}}

\newcommand\bG{\mathbb{G}}

\newcommand\bK{\mathbb{K}}

\newcommand\C{\mathcal{C}}

\newcommand\cA{\mathcal{A}}
\newcommand\cB{\mathcal{B}}

\newcommand\cD{\mathcal{D}}
\newcommand\cE{\mathcal{E}}

\newcommand\cH{\mathcal{H}}

\newcommand\cM{\mathcal{M}}
\newcommand\cN{\mathcal{N}}

\newcommand\cW{\mathcal{W}}

\newcommand\cZ{\mathcal{Z}}

\newcommand\sC{\mathscr{C}}

\newcommand\sE{\mathscr{E}}
\newcommand\sF{\mathscr{F}}
\newcommand\sG{\mathscr{G}}

\newcommand\sO{\mathscr{O}}

\newcommand\Y{\mathfrak{Y}}

\newcommand\Ho{\mathrm{Ho}}

\newcommand\Mod{\mathrm{Mod}}

\newcommand\Hom{\mathrm{Hom}}

\newcommand\HHom{\underline{\mathrm{Hom}}}

\newcommand\Ext{\mathrm{Ext}}

\newcommand\End{\mathrm{End}}

\newcommand\cone{\mathrm{cone}}
\newcommand\cocone{\mathrm{cocone}}
\newcommand\dg{\mathrm{dg}}
\newcommand\per{\mathrm{per}}

\newcommand\Ob{\mathrm{Ob}}

\newcommand\Ch{\mathrm{Ch}}

\newcommand\Co{\mathrm{Co}}

\newcommand\cpt{\mathrm{cpt}}

\newcommand\Spec{\mathrm{Spec}\,}

\newcommand\Cat{\mathrm{Cat}}

\newcommand\FD{\mathrm{FD}}

\newcommand\<{\langle}
\renewcommand\>{\rangle}

\newcommand\LLim{\varinjlim}
\DeclareMathOperator*{\holim}{holim}

\newcommand\into{\hookrightarrow}
\newcommand\onto{\twoheadrightarrow}

\newcommand\xra{\xrightarrow}
\newcommand\xla{\xleftarrow}

\newcommand\bt{\bullet}
\newcommand\by{\times}

\newcommand\Vect{\mathrm{Vect}}

\newcommand\Rep{\mathrm{Rep}}

\newcommand\CH{\mathrm{CH}}

\newcommand\et{\acute{\mathrm{e}}\mathrm{t}}

\newcommand\Nori{\mathrm{Nori}}

\newcommand\mot{\mathrm{mot}}

\newcommand\Tot{\mathrm{Tot}\,}

\newcommand\ind{\mathrm{ind}}

\newcommand\pd{\partial}

\newcommand\Sm{\mathrm{Sm}}

\newcommand\MHS{\mathrm{MHS}}

\newcommand\gr{\mathrm{gr}}

\newcommand\op{\mathrm{opp}}

\newcommand\eff{\mathrm{eff}}

\newcommand\co{\colon\thinspace}

\newcommand\oR{\mathbf{R}}

\newcommand\oL{\mathbf{L}}

\newcommand\uleft\underleftarrow
\newcommand\uline\underline
\newcommand\uright\underrightarrow

\newcommand\SmQP{\mathrm{SmQP}}
\newcommand\MT{\mathrm{MT}}
\newcommand\PMT{\mathrm{PMT}}
\newcommand\HS{\mathrm{HS}}

\newtheorem{theorem}{Theorem}[section]
\newtheorem{proposition}[theorem]{Proposition}
\newtheorem{corollary}[theorem]{Corollary}

\newtheorem{lemma}[theorem]{Lemma}
\newtheorem*{theorem*}{Theorem}
\newtheorem*{proposition*}{Proposition}
\newtheorem*{corollary*}{Corollary}
\newtheorem*{lemma*}{Lemma}
\newtheorem*{conjecture*}{Conjecture}

\theoremstyle{definition}
\newtheorem{definition}[theorem]{Definition}

\newtheorem*{definition*}{Definition}

\theoremstyle{remark}
\newtheorem{example}[theorem]{Example}

\newtheorem{remark}[theorem]{Remark}

\newtheorem*{example*}{Example}
\newtheorem*{examples*}{Examples}
\newtheorem*{remark*}{Remark}
\newtheorem*{remarks*}{Remarks}
\newtheorem*{exercise*}{Exercise}
\newtheorem*{property*}{Property}
\newtheorem*{properties*}{Properties}

\begin{document}

\begin{abstract}
 We develop Tannaka duality theory for dg categories. To any dg functor from a dg category $\cA$ to finite-dimensional complexes, we associate a dg coalgebra $C$ via a Hochschild homology construction. When the dg functor is faithful, this gives a quasi-equivalence between the derived dg categories of $\cA$-modules and of $C$-comodules. When $\cA$ is Morita  fibrant (i.e. an idempotent-complete pre-triangulated  category), it is thus quasi-equivalent to the derived dg category of compact $C$-comodules.  We give several applications for motivic Galois groups.
\end{abstract}

\title{Tannaka duality for enhanced triangulated categories I: reconstruction}
\author{J.P.Pridham}
\thanks{This work was supported by  the Engineering and Physical Sciences Research Council [grant numbers EP/I004130/1 and EP/I004130/2].}

\maketitle
\section*{Introduction}

Tannaka duality in Joyal and Street's formulation (\cite[\S 7, Theorem 3]{JoyalStreet}) characterises abelian $k$-linear categories $\cA$ with exact faithful $k$-linear functors $\omega$ to finite-dimensional $k$-vector spaces as categories of finite-dimensional comodules of coalgebras $C$. When $\cA$ is a rigid tensor category and $\omega$ monoidal, $C$ becomes a Hopf algebra (so $\Spec C$ is a group scheme), giving the duality theorem of \cite[Ch. II]{tannaka}.

The purpose of this paper is to extend these duality theorems to dg categories. Various derived versions of Tannaka duality have already been established, notably \cite{toentannaka,wallbridgeTannaka,FukuyamaIwanari, lurieDAG8h,iwanariTannakization}. However, those works usually require the presence of $t$-structures, and  all follow \cite[Ch. II]{tannaka} in restricting attention to  monoidal derived categories, 
  then take higher stacks as the derived generalisation of group schemes.

Our viewpoint does not require the dg categories to have monoidal structures, and takes dg coalgebras as the dual objects. Arbitrary dg coalgebras are poorly behaved (for instance, quasi-isomorphism does not imply Morita equivalence), but they perfectly capture the behaviour of arbitrary dg categories without $t$-structures.  Even in the presence of monoidal structures, we consider more general dg categories than heretofore, and our dg coalgebras then become dg bialgebras, in which case our results give dg enhancements and strengthenings of Ayoub's weak Tannaka duality from \cite{ayoubGaloisMotivic1}. A similar strengthening has appeared in  \cite{iwanariTannakization}, but  without the full description needed for applications to motives (see Remarks \ref{cfayoub0} and \ref{motayoub1}).  

The first crucial observation we make is that in the Joyal--Street setting, the dual coalgebra $C$ to $\omega\co \cA \to \FD\Vect$ is given by the Hochschild homology group
\[
 \omega^{\vee}\ten_{\cA}\omega = \HH_0(\cA, \omega^{\vee}\ten_k \omega),
\]
where $\omega^{\vee}\co \cA^{\op} \to \FD\Vect$ sends $X$ to the dual $\omega(X)^{\vee}$. 
The natural generalisation of the dual coalgebra to dg categories is then clear: given a $k$-linear dg category $\cA$ and a $k$-linear dg functor $\omega$ to finite-dimensional complexes, we put a dg coalgebra structure $C$ on the Hochschild homology complex
\[
 \omega^{\vee}\ten_{\cA}^{\oL}\omega\simeq \CCC_{\bt}(\cA, \omega^{\vee}\ten_k \omega).
\]

In order to understand the correct generalisation of the dg fibre functor $\omega$, we look to Morita (or Morita--Takeuchi) theory. In the underived setting, if $\omega$ is representable by an object $G\in \cA$, the condition that $\omega$ be exact and faithful amounts to requiring that $G$ be a projective generator for $\cA$. This means that in the dg setting, $\HHom(G,-)$ should be a dg fibre functor if and only if $G$ is a derived generator. In other words,  $\HHom(G,-)$ must reflect acyclicity of complexes, so we consider dg functors $\omega$ from $\cA$ to finite-dimensional complexes which are faithful in the sense  that $\omega(X)$ is acyclic only if $X$ is acyclic, for $X$ in the derived category $\cD(\cA)$. 

  Corollary \ref{tannakacor} gives a derived analogue of \cite[\S 7, Theorem 3]{JoyalStreet}. 
When $\omega$ is faithful, this gives a quasi-equivalence between the dg enhancements $\cD_{\dg}(\cA)$ and $\cD_{\dg}(C)$ of  the derived categories (of the first kind) of $\cA$ and $C$. This comparison holds for all dg categories; in particular, replacing $\cA$ with any subcategory of compact generators of $\cD_{\dg}(\cA)$ will yield a dg coalgebra $C$ with the same property. 
Our derived analogue of an abelian category is a Morita fibrant dg category: 
 when $\cA$ is such a dg  category, we have a quasi-equivalence between $\cA$ and the full dg subcategory of $\cD_{\dg}(C)$ on compact objects.

Crucially, Theorem \ref{tannakathm} gives a further generalisation to non-faithful dg functors $\omega$, showing that the    dg derived category $\cD_{\dg}(C)$ of $C$-comodules is quasi-equivalent to a derived quotient  $\cD_{\dg}(\cA)/(\ker\omega)$ of the 
dg derived category $\cD_{\dg}(\cA)$   generated by $\cA$. This has many useful applications to scenarios where $\cA$ arises as a quotient of a much simpler dg category $\cB$, allowing us to compute $C$ directly from $\cB$ and $\omega$.

Section \ref{HHsn} contains the key constructions used throughout the paper. After recalling the Hochschild homology complex $\CCC_{\bt}(\cA,F)$  of a dg category $\cA$ with coefficients in a $\cA$-bimodule $F$, we 
study the dg coalgebra $C_{\omega}(\cA):=\CCC_{\bt}(\cA, \omega^{\vee}\ten_k \omega)$. 

We then introduce the notion of   universal coalgebras of $\cA$, which are certain resolutions $D$ of $\cA(-,-)$ as a $\ten_{\cA}$-coalgebra in $\cA$-bimodules. The canonical choice is the Hochschild complex $\CCC_{\bt}(\cA, h_{\cA^{\op}}\ten_k h_{\cA})$ of the Yoneda embedding. For any universal coalgebra $D$, a dg fibre functor $\omega$ gives a dg coalgebra $C:= \omega^{\vee}\ten_{\cA}D\ten_{\cA}\omega$, and a tilting module $P:=D\ten_{\cA}\omega $. When $(\cA, \boxtimes)$ is a tensor dg category, we consider universal bialgebras, which are universal  coalgebras equipped with compatible multiplication with respect to $\boxtimes$, the Hochschild complex again being one such. In this case, a tensor dg functor $\omega$ makes $C$ into a dg bialgebra.
 
The main results of the paper are in Section \ref{comodsn}. For $C$ and $P$ a dg coalgebra and tilting module as above, there is a left Quillen dg functor $-\ten_{\cA}P$ from the category of dg $\cA$-modules to the category of dg $C$-comodules  (Lemma \ref{Qlemma}). The functors $\cD(C) \to \cD(\cA) \to \cD(C)$ then form a retraction (Proposition \ref{retractprop}). Theorem \ref{tannakathm}  establishes quasi-equivalences
\[
 \cD_{\dg}(C) \to  (\ker\omega)^{\perp}\to \cD_{\dg}(\cA)/(\ker \omega)
\]
of dg enhancements of derived categories, which simplify to the equivalences of Corollary \ref{tannakacor} when $\omega$ is faithful. Remark \ref{cfayoub0} relates this to Ayoub's weak Tannaka duality, with various consequences for describing motivic Galois groups  given  in \S \ref{mothomsn}. Proposition \ref{monoidalprop} ensures that the equivalences preserve tensor structures when present, and Example \ref{mothomb} applies this to motivic Galois groups. 

In the appendix, we give technical details for constructing  monoidal dg functors giving rise to the motivic Galois groups of Example \ref{mothomb}, and show that, in the case of  Betti cohomology, this construction gives a dg functor non-canonically quasi-isomorphic to the usual cohomology dg functor (Corollary \ref{bettiformalcor}).

The main drawback of the Hochschild construction for the dg coalgebra is that it always creates terms in negative cochain degrees. This means that quasi-isomorphisms of such dg coalgebras might not be derived Morita equivalences, and that we cannot rule out negative homotopy groups for dg categories of cohomological origin. This issue is addressed in the sequel \cite{HHtannaka2}, by relating vanishing of the dg coalgebra in negative degrees to the existence of a $t$-structure.  

I would like to thank Joseph Ayoub for providing helpful comments and spotting careless errors.


\subsection*{Notational conventions}

 Fix a commutative ring $k$.
 When the base is not specified, $\ten$ will mean $\ten_k$. When $k$ is a field, we write $\Vect_k$ for the category of all vector spaces over $k$, and $\FD\Vect_k$ for the full subcategory of finite-dimensional vector spaces.

We will always use the symbol $\cong$ to denote isomorphism, while $\simeq$ will be equivalence,  quasi-isomorphism or quasi-equivalence.

\tableofcontents

\section{Hochschild homology of a DG category}\label{HHsn}

\begin{definition}
 A $k$-linear dg category $\C$ is a category enriched in cochain complexes of $k$-modules, so has objects $\Ob \C$, cochain complexes $\HHom_{\C}(x,y)$ of morphisms, associative multiplication
\[
 \HHom_{\C}(y,z)\ten_k\HHom_{\C}(x,y)\to \HHom_{\C}(x,z)
\]
and identities $\id_x \in \HHom_{\C}(x,x)^0$.
\end{definition}

Given a dg category $\C$, we will write $\z^0\C$ and $\H^0\C$ for the categories with the same objects as $\C$ and with morphisms
\begin{align*}
 \Hom_{\z^0\C}(x,y) &:=\z^0\HHom_{\C}(x,y),\\
\Hom_{\H^0\C}(x,y) &:=\H^0\HHom_{\C}(x,y).
\end{align*}
 When we refer to limits or colimits in a dg category $\C$, we will mean limits or colimits in the underlying category $\z^0\C$.

\begin{definition}
Given a dg category $\C$ and objects $x,y$, write $\C(x,y):= \HHom_{\C}(y,x)$.
\end{definition}

\begin{definition}
 A dg functor $F\co \cA \to \cB$ is said to be a quasi-equivalence if $\H^0F\co \H^0\cA \to \H^0\cB$ is an equivalence of categories, with $\cA(X,Y) \to \cB(FX,FY)$ a quasi-isomorphism for all objects $X,Y \in \cA$.
\end{definition}

\begin{definition}\label{perkdef}
We follow \cite[2.2]{kellerModelDGCat} in writing $\C_{\dg}(k)$ for the dg category of chain complexes over $k$, where $\HHom(U,V)^i$ consists of graded $k$-linear morphisms $U \to V[i]$, and the differential is given by $df= d \circ f \mp f \circ d$. 

We write
$\per_{\dg}(k)$ for the full dg subcategory of 
finite rank  cochain complexes of projective $k$-modules. Beware that this category is not closed under quasi-isomorphisms, so does not include all perfect complexes in the usual sense. 
\end{definition}

The following is adapted from \cite[\S 12]{mitchellRings} and \cite[1.3]{kellerHCexact}:

\begin{definition}
Take a small $k$-linear dg category $\cA$ and an $\cA$-bimodule
\[
 F\co  \cA\by \cA^{\op}\to \C_{\dg}(k),       
\]
(i.e. a $k$-bilinear dg functor). Define the homological Hochschild complex
\[
 \uline{\CCC}_{\bt}(\cA, F)      
\]
(a simplicial diagram of cochain complexes) by
\[
 \uline{\CCC}_n(\cA,F):= \bigoplus_{X_0, \ldots, X_n \in \Ob \cA }  \cA(X_0,X_1)\ten_k\cA(X_1,X_2) \ten_k \ldots \ten_k\cA(X_{n-1},X_n) \ten_kF(X_n,X_0),   
\]
with face maps
\[
 \pd_i(a_1\ten\ldots a_n \ten f)= \left\{ \begin{matrix}   a_2\ten\ldots a_n \ten (f \circ a_1) & i=0 \\  a_1\ten\ldots a_{i-1} \ten (a_i \circ a_{i+1}) \ten a_{i+2} \ten \ldots\ten a_n \ten f & 0<i <n \\
  a_1\ten\ldots a_{n-1} \ten (a_n \circ f) & i=n                                              
                                          \end{matrix}\right.
 \]
and degeneracies
\[
 \sigma_i(a_1\ten\ldots a_n \ten f)= (a_1\ten \ldots \ten a_i \ten \id \ten a_{i+1} \ten \ldots \ten a_n \ten f).
 \]
\end{definition}

\begin{definition}


Define the total  Hochschild complex
\[
 \CCC(\cA,F)
\]
by first regarding $\uline{\CCC}_{\bt}(\cA, F)$ as a chain cochain complex with chain differential $\sum_i (-1)^i\pd_i$, then taking the total complex
\[
( \Tot \uline{\CCC}_{\bt}(\cA, F)^{\bt})^n = \bigoplus_i \uline{\CCC}_{i}(\cA, F)^{n+i},
\]
 with differential given by the cochain differential $\pm$ the chain differential. 

There is also a quasi-isomorphic normalised version 
\[
 N\CCC(\cA,F),
\]
given by replacing $\uline{\CCC}_{i}$ with $\uline{\CCC}_{i}/\sum_j \sigma_j \uline{\CCC}_{i-1}$.
\end{definition}

\begin{remark}
Note that $\H^i \CCC(\cA,F)^{\bt}= \HH_{-i}(\cA,F)$,  which is a Hochschild \emph{homology} group. We have, however, chosen cohomological gradings because our motivating  examples will all have $\H^{<0}=0$.
\end{remark}

\subsection{The Tannakian envelope}

Fix a small $k$-linear dg category $\cA$ and a $k$-linear dg  functor $\omega \co \cA \to \per_{\dg}(k)$.

\begin{remark}\label{hfdrmk}
If $k$ is a field and we instead have a dg functor $\omega\co  \cA \to h\FD\Ch_k$ to the dg category of cohomologically finite-dimensional complexes (i.e. perfect complexes in the usual sense), we can reduce to the setting above. We could first take a cofibrant replacement $\tilde{\cA} \to \cA$ of $\cA$ in Tabuada's model structure (\cite{tabuadaMCdgcat}, as adapted in \cite[Theorem 4.1]{kellerModelDGCat}) on dg categories. Because $k$ is a field,  the inclusion $\per_{\dg}(k) \to h\FD\Ch_k$ is a quasi-equivalence, so the composite dg functor $\omega \co \tilde{\cA} \to h\FD\Ch_k$ is homotopy equivalent to a dg functor $\omega' \co \tilde{\cA}\to \per_{\dg}(k)$.
\end{remark}

\begin{definition}\label{Comegadef}
 Define the Tannakian dual $C_{\omega}( \cA)$ by 
\[
 C_{\omega}( \cA):= \CCC(\cA,\omega \ten \omega^{\vee}),
\]
where the $\cA$-bimodule
\[
 \omega \ten \omega^{\vee} \co \cA \by \cA^{\op} \to \per_{\dg}(k) 
\]
is given 
by
\[
 \omega \ten \omega^{\vee} (x,y)= (\omega x)\ten_k (\omega y)^{\vee}.
\]
Similarly, write $NC_{\omega}( \cA):= N\CCC(\cA,\omega \ten \omega^{\vee})$.
\end{definition}

\begin{proposition}\label{coalg}
 The cochain complexes $C_{\omega}( \cA), NC_{\omega}( \cA) $ 
have the natural structure of  coassociative counital dg coalgebras over $k$.
\end{proposition}
\begin{proof}
We may rewrite
\begin{align*}
 & \uline{\CCC}_n(\cA,  {\omega}\ten\omega^{\vee})=\\
  & \bigoplus_{X_0, \ldots, X_n \in \Ob \cA }  \cA(X_0,X_1)\ten\cA(X_1,X_2) \ten \ldots \ten\cA(X_{n-1},X_n) \ten \omega X_n \ten \omega (X_0)^{\vee},  
\end{align*}

as
\[   \bigoplus_{X_0, \ldots, X_n \in \Ob \cA }  \omega (X_0)^{\vee} \ten\cA(X_0,X_1)\ten\cA(X_1,X_2) \ten \ldots \ten\cA(X_{n-1},X_n) \ten \omega X_n.     
\]

There is a comultiplication $\Delta$ on the bicomplex
\[
 \uline{\CCC}_{\bt}(\cA,\omega \ten \omega^{\vee}),
\]
with
\begin{align*}
 &  \Delta \co \uline{\CCC}_{m+n}(\cA,\omega \ten \omega^{\vee})\to \\
& \uline{\CCC}_m(\cA,\omega \ten \omega^{\vee})\ten_k\uline{\CCC}_n(\cA,\omega \ten \omega^{\vee})
\end{align*}
being  the map
\begin{align*}
 & ({\omega} X_0)^{\vee} \ten \C(X_0, X_{1}) \ten \ldots \ten \C(X_{m+n-1}, X_{m+n})\ten ({\omega} X_{m+n}) \to\\
&[ ({\omega} X_0)^{\vee} \ten \C(X_0, X_1) \ten \ldots \ten \C(X_{m-1}, X_m)\ten ({\omega} X_m)]\\
& \ten [({\omega} X_m)^{\vee} \ten \C(X_m, X_{m+1}) \ten \ldots \ten \C(X_{m+n-1}, X_{m+n})\ten ({\omega} X_{m+n})]
\end{align*}
 given by tensoring with
\[
 \id_X \in ({\omega} X_m) \ten ({\omega} X_m)^{\vee}.
\]

Now, 
\begin{align*}
&  (\pd_i\ten \id) \circ \Delta_{m+1,n} (x \ten c_1\ten \ldots \ten c_{m+n+1} \ten y)\\
 &=\left\{
\begin{matrix}
\Delta_{m,n} \pd_i  (x \ten c_1\ten \ldots \ten c_{m+n+1} \ten y)  & i\le m, \\
  x \ten c_1\ten \ldots \ten c_{m}\ten (\omega c_{m+1})\ten c_{m+2}\ten \dots \ten c_{m+n+1} \ten y &i=m+1;                                                            
\end{matrix}\right.\\
\end{align*}
\begin{align*}
 &(\id \ten \pd_i)\circ  \Delta_{m,n+1} (x \ten c_1\ten \ldots \ten c_{m+n+1} \ten y)\\
&= \left\{
\begin{matrix}
\Delta_{m,n} \pd_{i+m}  (x \ten c_1\ten \ldots \ten c_{m+n} \ten y)  & i>0, \\
  x \ten c_1\ten \ldots \ten c_{m}\ten (\omega c_{m+1})\ten c_{m+2}\ten \dots \ten c_{m+n+1} \ten y &i=0.                                                             
\end{matrix}\right.
\end{align*}
Thus the differential $d=\sum (-1)^i \pd_i$ has the property that
\[
[ (d\ten\id +(-1)^m\id \ten d) \circ \Delta]_{m,n}= \Delta_{m,n} \circ d.
\]

In other words, $\Delta$ is a chain map with respect to $d$, so passes to a comultiplication on $C_{\omega}( \cA) = \Tot \uline{\CCC}_{\bt}(\cA, \omega \ten \omega^{\vee})$. The properties of the $\pd_i$ above also ensure that the comultiplication descends to $NC_{\omega}( \cA)$.
\end{proof}

It is immediately clear that the constructions are functorial in the following sense:
\begin{lemma}\label{coalgfunlemma}
  For any $k$-linear dg functor $F \co \cB \to \cA$, there is an induced morphism
 $ C_{\omega\circ F}( \cB) \to C_{\omega}( \cA)$ of dg coalgebras, which also induces a morphism on the normalisations.
\end{lemma}

In \S \ref{funsn}, we will combine this lemma with  Theorem \ref{tannakathm} to show that $C_{\omega}( \cA)$ is essentially invariant under quasi-equivalent choices of $\cA$ and quasi-isomorphic choices of $\omega$, so that the choice in Remark \ref{hfdrmk} does not affect the output.

\subsection{The universal coalgebra and tilting modules}

\subsubsection{Background terminology}

Following the conventions of \cite[3.1]{kellerModelDGCat}, we will write $\C_{\dg}(\cA)$ for the  dg category  of $k$-linear dg functors $\cA^{\op} \to \C_{\dg}(k)$ to  chain complexes over $k$. Observe that when $\cA$ has a single object $*$ with $\cA(*,*)=A$, $\C_{\dg}(\cA)$ is equivalent to the category of $A$-modules in complexes. We write $\C(\cA)$ for the (non-dg) category $\z^0\C_{\dg}(\cA)$ of dg $\cA$-modules.

An object $P$ of  $\C(\cA)$ is cofibrant  (for the projective model structure) if every surjective quasi-isomorphism $L \to P$ has a section. The full dg subcategory of $\C_{\dg}(\cA)$ on cofibrant objects  is  denoted $\cD_{\dg}(\cA)$. 
This is the idempotent-complete  pre-triangulated category (in the sense of \cite[Definition 3.1]{BondalKapranov}) generated by $\cA$ and closed under filtered colimits. We write  $\cD(\cA)$ for the derived category $\H^0\cD_{\dg}(\cA)$ of dg $\cA$-modules --- this is equivalent to the localisation of $\C(\cA)$ at quasi-isomorphisms. Thus $\cD_{\dg}(\cA)$ is a dg enhancement of the triangulated category $\cD(A)$.

\begin{definition}\label{perdef}
 Define $\per_{\dg}(\cA) \subset \cD_{\dg}(\cA)$ to be the full subcategory on compact objects, i.e those $X$ for which
\[
 \HHom_{\cA}(X,-)
\]
 preserves filtered colimits.
Explicitly, $\per_{\dg}(\cA)$  consists of objects arising as direct summands of finite complexes of objects of the form $h_{X}[n]$, for $X \in \cA$, where $h$ is the Yoneda embedding. 
\end{definition}

When $\cA$ has a single object $*$ with $\cA(*,*)=A$, then $h_{*}[n]$ corresponds to the $A$-module $A[n]$. Since projective modules are direct summands of free modules, Definitions \ref{perdef} and \ref{perkdef} are thus consistent.

As explained in \cite[4.5]{kellerModelDGCat},  $\per_{\dg}(\cA)$ is the idempotent-complete   pre-triangulated envelope or hull of $\cA$, in the sense of \cite[\S 3]{BondalKapranov}. 
Note that in \cite[\S 2]{kellerHCexact}, pre-triangulated categories are called exact DG categories. 

By \cite[Theorem 5.1]{tabuadaInvariantsAdditifs}, there is a Morita model structure on $k$-linear dg categories. Weak equivalences are dg functors $\cA \to \cB$ which are derived Morita equivalences in the sense that
\[
 \cD_{\dg}(\cA) \to \cD_{\dg}(\cB)       
\]
is a quasi-equivalence.
The dg functor $\cA \to \per_{\dg}(\cA)$ is  fibrant replacement in this model structure.

Note that a dg category $\cA$ is an  idempotent-complete pre-triangulated category if and only if the natural embedding $\cA \to \per_{\dg}(\cA)$ is a quasi-equivalence. This is equivalent to saying that $\cA$ is Morita fibrant (i.e. fibrant in the Morita model structure), or triangulated in the terminology of \cite[Definition 2.4]{TVdg}.

\subsubsection{Universal coalgebras}\label{univcoalg}

\begin{definition}\label{monoidaldef}
Recall (e.g. from \cite[Remark 8.5]{toenMorita}) that there is a monoidal structure $\ten_{\cA}$ on the dg category $\C_{\dg}(\cA^{\op}\ten \cA)$, 
 given by 
\[
 (F\ten_{\cA}G)(X,Y)= F(X,-)\ten_{\cA}G(-,Y),      
\]
for $X \in \cA, Y \in \cA^{\op}$. The unit of the monoidal structure is the dg functor $\id_{\cA}$, given by
\[
 \id_{\cA}(X,Y)= \cA(X,Y).       
\]
\end{definition}

Take a $k$-linear dg category $\cA$, and  $D \in \cD_{\dg}(\cA^{\op}\ten \cA)$  a coassociative $\ten_{\cA}$-coalgebra, with the co-unit $D \to \id_{\cA}$
a  quasi-isomorphism. We regard this as being a universal coalgebra associated to $\cA$.

\begin{example}\label{HHcoalgex}
If the $k$-complexes $\cA(X, Y)$ are all cofibrant (automatic when $k$ is a field), a canonical choice for $D$ is the Hochschild complex
\[
   \CCC(\cA,h_{\cA^{\op}} \ten h_{\cA})     
\]
of the Yoneda embedding $h_{\cA^{\op}} \ten h_{\cA}  \co \cA^{\op}\ten \cA\to \C_{\dg}(\cA^{\op}\ten \cA)$. Explicitly, $\CCC(\cA,h_{\cA^{\op}} \ten h_{\cA})$ is the total complex of the chain complex
\[
 \bigoplus_{X_0 \in \cA} X_0^{\op} \ten X_0 \la   \bigoplus_{X_0,X_1 \in \cA} X_0^{\op} \ten \cA(X_0, X_1)\ten X_1    \la \ldots, 
\]
where we write $X$ and  $X^{\op}$ for the images of $X$ under the Yoneda embeddings $\cA\to \C_{\dg}( \cA)$,  $\cA^{\op}\to \C_{\dg}(\cA^{\op})$.

The $\ten_{\cA}$-coalgebra structure is given by the formulae of Proposition \ref{coalg}, noting that 
\[
 Y\ten_{\cA}X^{\op}= \cA(Y,X),       
\]
so $\id_X \in X\ten_{\cA}X^{\op}$.

The normalised version of the Hochschild complex $N\CCC(\cA,h_{\cA^{\op}} \ten h_{\cA})$ provides another choice for $D$, which is more canonical in some respects.

 If we write $L= \bigoplus_{X \in \cA} X^{\op}\ten_k X$, then $L$ is a $\ten_{\cA}$-coalgebra in $\C_{\dg}(\cA^{\op}\ten \cA)$. The counit is just the composition $\bigoplus_X \cA(-,X)\ten_k \cA(X,-) \to  \cA(-,-)$, and comultiplication comes from $\id_X \in  X\ten_{\cA}X^{\op}$. Then $D$ is the total complex of the simplicial diagram given by $ \underbrace{L\ten_{\cA}L \ten_{\cA} \ldots \ten_{\cA}L}_{n+1}$ in level $n$,  so $D$ is just the \v Cech nerve of the $\ten_{\cA}$-comonoid $L$. 
\end{example}

\begin{definition}
Say that a coassociative $\ten_{\cA}$-coalgebra $D \in \cD_{\dg}(\cA^{\op}\ten \cA)$ is ind-compact if it can be expressed as a filtered colimit $D\cong \LLim_i D_i$ in the underlying category  $\z^0\cD_{\dg}(\cA^{\op}\ten \cA)$, with each $D_i$ a coassociative $\ten_{\cA}$-coalgebra which is compact as an object of $\cD_{\dg}(\cA^{\op}\ten \cA)$.
       \end{definition}

Note that if $\cA$ itself is a field, then the fundamental theorem of coalgebras (\cite[Proposition 7.1]{JoyalStreet}) says that all $\ten_{\cA}$-coalgebras are filtered colimits (i.e. nested unions) of finite-dimensional coalgebras, so are ind-compact. 

\begin{example}\label{HHunivcoalg}
If $k$ is a field, then the $\ten_{\cA}$-coalgebra $\CCC(\cA,h_{\cA^{\op}} \ten h_{\cA})$ is ind-compact. We construct the exhaustive system of compact subcoalgebras as follows. The indexing set will consist of triples $(S, n,V)$ with $S$ a finite subset of $\Ob \cA$, $n \in \N_0$ and $V(X,Y) \subset \cA(X,Y)$ a collection of finite-dimensional cochain complexes for $X,Y \in S$.

For $X',Y' \in S$, we now let $V^{(i)}(X',Y')\subset \cA(X',Y')$ be the cochain complex generated by strings of length at most $2^i$ in elements of $V$. 
We now define
$
 D_{(S,n,V)} \subset   D     
$
to be the total complex of
\begin{align*}
 \bigoplus_{X_0 \in S} X_0^{\op} \ten X_0 \la  & \bigoplus_{X_0,X_1 \in S} X_0^{\op} \ten V^{(n-1)}(X_0, X_1)\ten X_1    \\
 \la &\bigoplus_{X_0,X_1,X_2 \in S} X_0^{\op} \ten V^{(n-2)}(X_0, X_1)\ten V^{(n-2)}(X_1, X_2) \ten X_2   \\ 
\la\ldots\la & \bigoplus_{X_0,\ldots X_n \in S} X_0^{\op} \ten V^{(0)}(X_0, X_1)\ten\ldots \ten V^{(0)}(X_{n-1}, X_n) \ten X_n.       
\end{align*}

This is indeed a complex because multiplication in $\cA$ gives boundary maps $V^{(n-i)}(X,Y)\ten V^{(n-i)}(Y,Z)\to V^{(n-i-1)}(X,Z)$, and it is a subcoalgebra because $V^{(n-i-j)} \subset V^{(n-i)}\cap V^{(n-j)}$. The indexing set becomes a poset by saying $(S,m,U)\subset (T,n,V) $ whenever $S \subset T$, $m \le n$ and $U \subset V$. Thus we have a filtered colimit
\[
 D= \LLim_{(S,n,V)}D_{(S,n,V)},       
\]
of the required form.
\end{example}


\subsubsection{Tilting modules}\label{tilt}

Given 
$\omega \co \cA \to \per_{\dg}(k)$, define the tilting module  $P$ by $P:=D\ten_{\cA}\omega \in \C(\cA^{\op})$; this is  cofibrant and has  a natural quasi-isomorphism $P \to \omega$. Also set $Q\in \C(\cA)$ by $Q:=\omega^{\vee}\ten_{\cA}D$ and set $C:= \omega^{\vee}\ten_{\cA}D\ten_{\cA}\omega$. Note that the natural transformation $\id_{\cA} \to \omega\ten_k \omega^{\vee}$ makes $C$ into a dg coalgebra over $k$:
\begin{align*}
 C= \omega^{\vee}\ten_{\cA}D\ten_{\cA}\omega&\to \omega^{\vee}\ten_{\cA}D\ten_{\cA}D\ten_{\cA}\omega\\
&=\omega^{\vee}\ten_{\cA}D\ten_{\cA}\id_{\cA}\ten_{\cA}D\ten_{\cA}\omega\\
&\to \omega^{\vee}\ten_{\cA}D\ten_{\cA}\omega\ten_k \omega^{\vee}\ten_{\cA}D\ten_{\cA}\omega\\
&=C\ten_kC.
\end{align*}
Likewise,  $P$ becomes a right $C$-comodule and $Q$ a left $C$-comodule.

Also note that because $D$ is a cofibrant replacement for $\id_{\cA}$, we have
\[
C\simeq  \omega^{\vee}\ten_{\cA}^{\oL}\id_{\cA}\ten_{\cA}^{\oL}\omega \simeq \omega^{\vee}\ten^{\oL}_{\cA}\omega.
\]

For a chosen exhaustive system $D= \LLim_i D_i$ of an ind-compact coalgebra, we also write $P_i:=D_i\ten_{\cA}\omega$, $C_i:= \omega^{\vee}\ten_{\cA}D_i\ten_{\cA}\omega$ and $ Q_i:=\omega^{\vee}\ten_{\cA}D_i$. Each $C_i$ is a dg coalgebra, with $P_i$ (resp. $Q_i$) a right (resp. left) $C_i$-comodule.

\begin{example}
When $D= \CCC(\cA,h_{\cA^{\op}} \ten h_{\cA})$, observe that 
\begin{align*}
 C &= \CCC(\cA, \omega \ten \omega),\\
P&= \CCC(\cA, h_{\cA} \ten \omega),\\
Q&=\CCC(\cA, \omega^{\vee} \ten h_{\cA^{\op}}),
\end{align*}
so $C$ is just the dg coalgebra $C_{\omega}(\cA)$ of Definition \ref{Comegadef}. 

\end{example}

\subsubsection{Preduals}

\begin{definition}\label{predualdef}
Given $M \in \C_{\dg}(\cA)$, define the predual $M'\in \C_{\dg}(\cA^{\op})$ to be the dg functor $M'\co \cA \to \C_{\dg}(k)$  given by $M'(Y)= \HHom_{\C_{\dg}(\cA)}(M, Y)$. Note that this construction is only quasi-isomorphism invariant for $M \in \cD_{\dg}(\cA)$. 
\end{definition}

Observe that for $X \in \cA$ and the Yoneda embeddings $h$, we have $h_X'= h_{X^{\op}}$, giving  isomorphisms $N\ten_{\cA}h_X \cong N(X) \cong \HHom_{\cA^{\op}}(h_X',N)$ for all $N \in \C_{\dg}(\cA^{\op})$. Passing to finite complexes and arbitrary colimits in $\C_{\dg}(\cA)$, this gives us a natural transformation
\[
 N\ten_{\cA}M \to \HHom_{\cA^{\op}}(M',N)
\]
for all $M \in \C_{\dg}(\cA)$; this is necessarily an isomorphism when $M \in \per_{\dg}(\cA)$ because both sides preserve finite complexes and direct summands.

\subsection{Monoidal categories}

In order to recover the setting of \cite[Ch. II]{tannaka}, we now introduce monoidal structures. For the purposes of this subsection $(\cA,\boxtimes,\mathbbm{1})$ is a strictly  monoidal dg category, so we have $k$-linear dg functors $\mathbbm{1} \co k \to \cA$ and  $\boxtimes \co \cA \ten \cA \to \cA$, such that if we also write $\mathbbm{1}$ for the image of the unique object in $k$,
\[
 (X\boxtimes Y)\boxtimes Z = X \boxtimes (Y \boxtimes Z), \quad \mathbbm{1} \boxtimes X = X, \quad X \boxtimes \mathbbm{1} = \mathbbm{1}.
\]

\begin{definition}
 Say that a dg functor $\omega\co \cA \to \per_{\dg}(k)$  is lax monoidal if it is equipped with natural transformations
\[
 \mu_{XY} \co \omega(X)\ten \omega(Y)\to \omega(X\boxtimes Y), \quad  \eta \co k \to \omega(\mathbbm{1})
\]
satisfying associativity and unitality conditions.

It is said to be strict (resp. strong, resp. quasi-strong) if $\mu$ and $\eta$ are equalities (resp. isomorphisms, resp. quasi-isomorphisms).
\end{definition}

\begin{remark}
The hypothesis that $\cA$ and $\omega$ be strictly monoidal is of course very strong. All the results of this section will be straightforwardly functorial with respect to isomorphisms, though not always with respect to quasi-isomorphisms, so we could replace $\cA$ with any equivalent dg category (quasi-equivalent does not suffice). Thus the results will also apply to strongly monoidal dg categories and functors, where the equalities above are replaced by isomorphisms in such a way that $\z^0\cA$ becomes a strongly monoidal category and $ \omega \co \z^0\cA \to \z^0\per_{\dg}(k)$ a strongly monoidal functor. This condition is satisfied by Example \ref{mothomb}, our main motivating example. 
\end{remark}

\subsubsection{The Tannakian envelope for strongly monoidal functors}\label{monoidalenv}

\begin{definition}
 Given dg functors $F \co \cB \to \per_{\dg}(k) $, $G \co \C \to \per_{\dg}(k) $, define the external tensor product
\[
 F\odot G \co \cB \ten \C \to \per_{\dg}(k) 
\]
by $(F\odot G)(X\ten Y):= F(X)\ten_k G(Y)$. 
\end{definition}

\begin{lemma}\label{shufflelemma}
 For dg categories $\cB,\C$ and $k$-linear dg functors $F \co \cB \to \per_{\dg}(k) $,  $G \co \C \to \per_{\dg}(k) $, the dg coalgebras of Proposition \ref{coalg} have canonical quasi-isomorphisms
\begin{align*}
 C_F(\cB) \ten_k C_G(\C) &\to C_{F\odot G} (\cB\ten \C),\\
NC_F(\cB) \ten_k NC_G(\C) &\to NC_{F\odot G} (\cB\ten \C).
\end{align*}
These maps are symmetric on interchanging $(\cB,F)$ and $(\C,G)$, and the construction is associative in the sense that it induces a unique map 
\[
 C_F(\cB) \ten_k C_G(\C)\ten_k C_H(\cD)\to C_{F\odot G\odot H} (\cB\ten \C\ten \cD),
\]
and similarly for $NC$.
\end{lemma}
\begin{proof}
First observe that we have  canonical isomorphisms
\[
\uline{\CCC}_m(\cB\ten \C , (F\odot G)^{\vee}\ten (F\odot G) ) \cong \uline{\CCC}_m(\cB,F) \ten_k   \uline{\CCC}_m(\C , G)
\]
for all $m$. These isomorphisms are clearly compatible with the comultiplication  maps $\Delta\co \uline{\CCC}_{m+n}\to \uline{\CCC}_m\ten\uline{\CCC}_n$ and with the simplicial operations.

Now the Eilenberg--Zilber shuffle product of   \cite[I.4.2--3 ]{QRat} gives a symmetric associative quasi-isomorphism from $C_F(\cB) \ten_k C_G(\C)$ to the total complex of the simplicial cochain complex $m \mapsto \uline{\CCC}_m(\cB,F) \ten_k   \uline{\CCC}_m(\C , G)$, which is compatible with the comultiplications.  Combined with the isomorphisms above, this gives
\[
 C_F(\cB) \ten_k C_G(\C) \to C_{F\odot G} (\cB\ten \C),
\]
and similarly on normalisations. 
\end{proof}

\begin{proposition}\label{bialg}
If $\omega \co \cA \to \per_{\dg}(k) $ is strongly monoidal, the monoidal structures endow the   dg coalgebras $C_{\omega}( \cA), N C_{\omega}( \cA)$  with the natural structure of unital dg bialgebras. These are graded-commutative whenever $\boxtimes$ and $\omega$ are symmetric.
\end{proposition}
\begin{proof}
Define a dg functor $\boxtimes_*\omega$ on $\cA \ten \cA$ by $(\boxtimes_*\omega)(X\ten Y) := \omega(X \boxtimes Y)$. 
 We may  apply Lemma \ref{coalgfunlemma} to the dg functor $\boxtimes$ to obtain a morphism
\[
C_{\boxtimes_*\omega}(\cA \ten \cA) \to  C_{\omega}(\cA). 
\]

Strong monoidality of $\omega$ gives $\boxtimes_*\omega \cong \omega \odot \omega$ and hence $C_{\boxtimes_*\omega}(\cA \ten \cA) \cong C_{\omega\odot \omega}(\cA \ten \cA)$. Lemma \ref{shufflelemma} then provides a dg coalgebra quasi-isomorphism 
\[
 C_{\omega}( \cA) \ten_k C_{\omega}( \cA) \to C_{\omega\odot \omega}(\cA \ten \cA),
\]
which completes the construction of the associative multiplication. This product is moreover commutative whenever $\boxtimes$ and $\omega$ are symmetric, and induces a product on  $NC_{\omega}( \cA)$ similarly.

Applying Lemma \ref{coalgfunlemma} to the  unit $\mathbbm{1}$ similarly induces  morphisms 
\[
k= C_{\id}(k) \cong  C_{\omega \circ \mathbbm{1}}(k) \to C_{\omega}(\cA),
\]
and unitality of $\omega$ and $\mathbbm{1}$ ensures that this is a unit for the multiplication above.
\end{proof}

\begin{remark}\label{antipodermk}
 In the scenario considered in \cite[Ch. II]{tannaka}, the tensor category  was rigid in the sense that it admitted strong duals, or equivalently internal $\Hom$s. Then the Tannaka dual bialgebra $\HH_0(\cA, \omega^{\vee}\ten_k\omega)$  became a Hopf algebra. 

If our dg  category $\cA$ has strong duals, then we may define an involution $\rho$ on $C_{\omega}(\cA)$ by combining the isomorphism $C_{\omega}(\cA)^{\op} \cong C_{\omega^{\vee}}(\cA^{\op})$ with the map $C_{\omega^{\vee}}(\cA^{\op}) \to C_{\omega}(\cA) $ induced by applying  Lemma \ref{coalgfunlemma} to the duality functor.

The condition that $\rho$ be an antipode on a bialgebra $C$ is that the diagrams
\[
 \begin{CD}
  C @>{\Delta}>> C\ten C\\
@V{\vareps}VV @VV{(\rho \ten \id)}V \\
k @>1>> C
 \end{CD} \quad\quad \quad
\begin{CD}
  C @>{\Delta}>> C\ten C\\
@V{\vareps}VV @VV{(\id \ten \rho )}V \\
k @>1>> C
 \end{CD} 
\]
commute.

On the bialgebra $\pi_0\uline{\CCC}_{\bt}(\cA, \omega^{\vee}\ten \omega)$, it turns out that $\rho$ defines an antipode, making the bialgebra into a Hopf algebra and recovering the construction of \cite[II.2]{tannaka}. However, the  dg bialgebras $C_{\omega}( \cA), N C_{\omega}( \cA)$ are far from being dg Hopf algebras.
This is easily seen by looking at 
\[
 \uline{\CCC}_0(\cA, \omega^{\vee}\ten \omega)^{\vee}= \prod_{X \in \cA} \End(\omega X).
\]
The antipodal condition above reduces to saying that for all $f \in \uline{\CCC}_0(\cA, \omega^{\vee}\ten \omega)^{\vee}$, we require that $\omega(\eps_X)\circ  f_{X^*\boxtimes X}= f_1 \circ \omega(\eps_X)$, for $\eps_X \co X^*\boxtimes X\to {\mathbbm{1}}$ the duality transformation. 
 There are few dg categories $\cA$ for which this holds, so $\rho$ seldom makes $\uline{\CCC}_0(\cA, \omega^{\vee}\ten \omega)$ into a Hopf algebra. However, the condition above automatically holds for all Hochschild $0$-cocycles, which is why  $\pi_0\uline{\CCC}_{\bt}(\cA, \omega^{\vee}\ten \omega)$ is a Hopf algebra.

In the sequel \cite[\S \ref{HHtannaka2-tensorsubsn}]{HHtannaka2} there appears a context where a variant of the Hochschild complex does have a suitable antipode, and hence the structure of a Hopf algebra.
\end{remark}

\subsubsection{The Tannakian envelope for lax monoidal functors}\label{laxsn}

If the dg  functor $\omega$ is only quasi-strong, it is too much to expect that $C_{\omega}(\cA)$ will be a bialgebra in general. A dg bialgebra is a monoid in dg coalgebras,  and it is not usually possible to strictify algebraic and coalgebraic structures simultaneously, so $C_{\omega}(\cA)$ should be a form of strong homotopy monoid in dg coalgebras.

We will now construct the structures enriching $C_{\omega}(\cA)$ for any lax monoidal dg functor $\omega$. When $\omega$ is quasi-strong, Corollary \ref{QIMcor} will ensure that this indeed gives a form of strong homotopy monoid.

\begin{definition}
Define $I$ to be the category on two objects $0,1$ with a unique non-identity morphism $ 0 \to 1$. Define $K$ to be the category whose objects are $I, \{0\},\{1\}$ and whose non-identity morphisms are the inclusions $\{0\}, \{1\} \to I$. Thus the objects of $K$ are categories in their own right. 
\end{definition}

As in \cite[Definition \ref{monad-delta**}]{monad}, we will write $\Delta_{**}$ for the subcategory of the ordinal number category $\Delta$ consisting of morphisms which fix the initial and final vertices. This has a monoidal structure given by setting $\mathbf{m}\ten \mathbf{n}= \mathbf{m+n}$. As in \cite[Definition \ref{monad-cflein1}]{monad}, $\Delta_{**}$ is opposite to the augmented ordinal number category, so an up-to-homotopy monoid structure on $C \in\C$ in the sense of 
\cite{leinster} is a colax monoidal functor $M\co \Delta_{**} \to \C$ for which the maps $M(i+j)\to M(i)\ten M(j)$ are weak equivalences and $M(\mathbf{1})=C$.

We now adapt \cite[Definition \ref{monad-Xidef}]{monad}:
\begin{definition}
 Define $\bK \co \Delta_{**} \to \Cat$ to be the category-valued lax monoidal functor given on objects by  $\mathbf{0} \mapsto *$,  $\mathbf{n}\mapsto K^{n-1}$ and on morphisms by
\begin{eqnarray*}
\bK(\pd^i)(k_1,\ldots,k_n)&=&(k_1, \ldots,k_{i-1},\{1\},k_i,\ldots,k_n);\\ 
\sigma^i(k_1,\ldots,k_n)&=&(k_1, \ldots,k_{i-1}, m(k_i,k_{i+1}),k_{i+2},\ldots, k_n);\\
\sigma^0(k_1,\ldots,k_n)&=&(k_2, \ldots,k_n);\\
\sigma^{n}(k_1,\ldots,k_n)&=&(k_1, \ldots,k_{n-1}),
\end{eqnarray*}
where $m \co K^2 \to K$ is the symmetric functor determined by $m(\{0\}, k)= \{0\}$, $m(\{1\}, k)= k$, $m(I,I)=I$. The monoidal structure on $\bK$ is given by the maps
\[
 K^{m-1} \by K^{n-1} \cong K^{m-1} \by\{0\} \by  K^{n-1}\subset K^{m+n-1}.
\]
\end{definition}

\begin{proposition}\label{weakbialg} 
If $\omega \co \cA \to \per_{\dg}(k) $ is lax monoidal with $\omega(\mathbbm{1})=k$, the monoidal structures give rise to  a colax monoidal functor $\bC_{\omega}(\cA)$ from the opposite $\Gamma(\bK^{\op})^{\op}$ of the Grothendieck construction $\Gamma(\bK^{\op})$ of $\bK^{\op}$ to the category of dg coalgebras over $k$, with $\bC_{\omega}(\cA)(\mathbf{1})= C_{\omega}(\cA)$. There is a similar construction for $NC_{\omega}(\cA)$.
\end{proposition}
\begin{proof}
We prove this for  $C_{\omega}(\cA)$, the proof for $NC_{\omega}(\cA)$ being entirely similar. On the category $\cA^{\ten 2}$, $\omega$ induces dg fibre functors $\omega\odot \omega$ and $ \boxtimes_* \omega$. The transformation between them gives a dg fibre functor $\uline{\omega}$ on $\cA^{\ten 2} \by I$ with $(X,Y,0) \mapsto \omega(X)\ten \omega(Y)$ and $(X,Y,1) \mapsto \omega(X\boxtimes Y)$. 

Iterating this construction gives us dg fibre functors $\uline{\omega} \co \cA^{\ten n} \by I^{n-1}\to \per_{\dg}(k) $ for all $n$. At a vertex of $I^{n-1}$, the corresponding fibre functor is given by writing $\omega$ for each $0$ co-ordinate and $\boxtimes_*$ for each $1$ co-ordinate, then appending a final $\omega$ and introducing  $\odot$s as separators. 
For $I^{n-1} \in K^{n-1}$, we define
\[
 \bC_{\omega}(\cA)(n,I^{n-1}):= C_{\uline{\omega}}(\cA \by I^{n-1}). 
\]

Now, any object $k \in K^{n-1}$ is a subcategory of $I^{n-1}$, and we may define $ \bC_{\omega}(\cA)(n,k):= C_{\uline{\omega}|_k}(\cA \by k)$. Lemma \ref{coalgfunlemma} then gives morphisms $ \bC_{\omega}(\cA)(n,k)\to \bC_{\omega}(\cA)(n,l)$ for each morphism $k \to l$ in $K^{n-1}$. This defines $\bC_{\omega}(\cA)$ on the subcategories $\bK(n) \subset \Gamma(\bK^{\op})^{\op}$, and it remains to define the images of the cosimplicial morphisms $\pd^i, \sigma^i$ and the monoidal structure.

The fibred dg functor  $(\cA\ten \cA, \omega \boxtimes \omega) \xra{\boxtimes} (\cA, \omega)$ induces fibred dg functors $(\cA^{\ten n+1} \by \pd^i(k), \uline{\omega}|_{\pd^i(k)}) \to (\cA^{\ten n} \by k, \uline{\omega}|_k)$ for all $i,n$ and $k \in K^{n-1}$. By Lemma \ref{coalgfunlemma}, this gives a morphism
\[
 \pd_i \co \bC_{\omega}(\cA)(n+1, \pd^i(k)) \to \bC_{\omega}(\cA)(n, k)
\]
of dg coalgebras, which we define to be the image of $\pd_i \co (\mathbf{n+1}, \pd^i(k)) \to (\mathbf{n}, k)$.

Substituting the unit $\mathbbm{1} \in \cA$ in either factor induces fibred dg functors $(\cA, \omega) \to (\cA\ten \cA, \omega \boxtimes \omega)$ and $(\cA, \omega) \to (\cA\ten \cA, \omega \odot \omega)$. Similar arguments show that these induce morphisms
\[
 \sigma_i \co \bC_{\omega}(\cA)(n-1, \sigma^i(k)) \to \bC_{\omega}(\cA)(n, k).
\]

To define the colax monoidal structure, we need morphisms 
\[
\bC_{\omega}(\cA)(m,k) \ten \bC_{\omega}(\cA)(n,l)  \to  \bC_{\omega}(\cA)(m+n, k\by \{0\} \by l),
\]
but these are just given by Lemma \ref{shufflelemma}.
\end{proof}

\begin{remark}
We can give $\Gamma(\bK^{\op})$ the structure of a relative category by setting a  morphism to be a weak equivalence when its image in $\Delta_{**}$ is the identity.  Since $K$ has a final object, its nerve is contractible, so the projection map $\Gamma(\bK^{\op}) \to \Delta_{**} $ is a weak equivalence of relative categories in the sense of \cite{barwickkanrelative}. 

If $\omega$ is quasi-strong, then Corollary \ref{QIMcor} will imply that $\bC_{\omega}$  sends all weak equivalences to derived Morita equivalences. If we let $\C$ be the relative category of dg $k$-coalgebras with weak equivalences given by derived Morita equivalences, then Proposition \ref{weakbialg} gives a   colax monoidal functor $ (\int\bK^{\op})^{\op}   \to \C$ of relative categories. By contrast, an  up-to-homotopy monoid is just a colax monoidal functor $\Delta_{**}^{\op} \to \C$  of relative categories. Since $\int\bK^{\op} $ is weakly equivalent to $\Delta_{**} $,  Proposition \ref{weakbialg} can thus be regarded as giving $C_{\omega}(\cA)$ the structure of a monoid up to coherent homotopy whenever $\omega$ is quasi-strong.

If the monoidal structure on the pair $(\cA,\omega)$ is moreover symmetric, then structures above can be adapted by replacing $\Delta$ with the category of finite sets, thus incorporating symmetries by not restricting to non-decreasing maps, and giving rise to a homotopy coherent symmetric monoid.
\end{remark}

\subsubsection{The universal bialgebra}\label{univbialg}

The monoidal structure $\boxtimes$ on $\cA$ induces a monoidal structure on $\cA^{\op}$, which we also denote by $\boxtimes$. There is also a monoidal structure $\boxtimes^2$ on $\cA^{\op}\ten \cA$, given by $(X\ten Y)\boxtimes^2(X'\ten Y'):= (X\boxtimes X')\ten (Y\boxtimes Y')$.

\begin{definition}\label{boxtimesdgdef}
 Define the dg functor 
\[
 \boxtimes \co \cD_{\dg}(\cA) \ten \cD_{\dg}(\cA) \to \cD_{\dg}(\cA)
\]
 as follows. The dg functor $\boxtimes \co \cA \ten \cA \to \cA$ induces a dg functor $\cD_{\dg}(\cA\ten \cA) \to \cD_{\dg}(\cA)$, which we compose with the dg functor $ \cD_{\dg}(\cA) \ten \cD_{\dg}(\cA) \xra{\odot} \cD_{\dg}(\cA\ten \cA) $ given by $(M\odot N)(X \ten Y):= M(X) \ten_k N(Y)$ for $X, Y \in \cA$.
    Define  $\boxtimes \co \cD_{\dg}(\cA^{\op}) \ten \cD_{\dg}(\cA^{\op}) \to \cD_{\dg}(\cA^{\op})$ and   $\boxtimes^2 \co \cD_{\dg}(\cA^{\op} \ten \cA) \ten \cD_{\dg}(\cA^{\op} \ten \cA) \to \cD_{\dg}(\cA^{\op} \ten \cA)$ similarly.
\end{definition}
In simpler terms, $\boxtimes$ is just given by extending the dg functor on $\cA$ to finite complexes, filtered colimits and direct summands.

\begin{definition}
 Define $\boxtimes_* \co \C_{\dg}(\cA) \to \C_{\dg}(\cA \ten \cA)$ by setting
\[
 (\boxtimes_*M)(X\ten Y) = M(X \boxtimes Y)
\]
for $X,Y \in \cA$.

Define $\boxtimes_* \co \C_{\dg}(\cA^{\op}) \to \C_{\dg}(\cA^{\op} \ten \cA^{\op})$ and $\boxtimes_*^2 \co \C_{\dg}(\cA^{\op}\ten \cA) \to \C_{\dg}(\cA^{\op}\ten \cA \ten \cA^{\op}\ten \cA)$ similarly.
\end{definition}

\begin{remark}\label{boxadjoint}
 For $S,T \in \cD_{\dg}(\cA)$ and $M \in \C_{\dg}(\cA)$, note that we have a natural isomorphism
\[
 \HHom_{\C_{\dg}(\cA)}(S \boxtimes T, M) \cong \HHom_{\C_{\dg}(\cA\ten \cA)}(S\ten_k T, \boxtimes_*M).
\]
 This isomorphism is tautological when $S=h_X,T=h_Y$ for $X,Y \in \cA$, noting that 
\[
 h_{X\ten Y}(U\ten V)= \cA(X,U)\ten_k \cA(Y,V)= h_X\ten_k h_Y.
\]
The general case follows by passing to complexes and direct summands.

The same observation holds for any monoidal dg category, and hence to $(\cA^{\op}, \boxtimes)$ and $(\cA^{\op}\ten \cA, \boxtimes^2)$.
\end{remark}

\begin{lemma}
The unit $\id_{\cA} \in \C(\cA\ten \cA^{\op})$ is equipped with a canonical associative multiplication
\[
 \id_{\cA}\ten_k \id_{\cA} \to \boxtimes_*^2\id_{\cA},
\]
 which is commutative whenever $\boxtimes$ is symmetric. The unit for this multiplication is 
\[
 \id_{\mathbbm{1}} \in \cA({\mathbbm{1}},{\mathbbm{1}})= \id_{\cA}({\mathbbm{1}},{\mathbbm{1}})= {\mathbbm{1}}^2_*k,
\]
for ${\mathbbm{1}}^2_*\co \C_{\dg}(k) \to \C_{\dg}(\cA^{\op}\ten \cA)$.
\end{lemma}
\begin{proof}
Evaluated at $X\ten Y\ten X'\ten Y' \in \cA\ten \cA^{\op}\ten \cA\ten \cA^{\op}$, 
this is just the map 
\[
 \cA(X,Y)\ten_k \cA(X',Y') \to \cA(X\boxtimes X', Y\boxtimes Y')
\]
induced by the bilinearity of $\boxtimes$.
\end{proof}

\begin{definition}\label{univbialgdef}
 We say that a universal coalgebra $D$ (in the sense of \S \ref{univcoalg}) is a universal bialgebra with respect to $\boxtimes$ if is equipped with an associative multiplication $D\ten_kD \to \boxtimes_*^2D $ and unit $k \to D({\mathbbm{1}},{\mathbbm{1}})$. These are required to be compatible with the coalgebra structure, in the sense that the comultiplication and co-unit
\[
 D\to D\ten_{\cA}D \quad D \to \id_{\cA}
\]
must be morphisms of associative unital $\boxtimes^2$-algebras.
 
When $\boxtimes$ is symmetric, we say that $D$ is a universal commutative bialgebra if the multiplication $D\ten_kD \to \boxtimes_*^2D $ is commutative.
\end{definition}

\begin{remark}
 Since universal coalgebras are required to be objects of $\cD_{\dg}(\cA^{\op}\ten \cA)$, we may apply Remark \ref{boxadjoint} to rephrase the algebra structure on $D$ to be an associative multiplication $D\boxtimes^2D \to D$ and a unit ${\mathbbm{1}}\ten {\mathbbm{1}} \to D$. 
\end{remark}

\begin{example}\label{HHunivbialg}
Under the conditions of Example \ref{HHunivcoalg} (e.g. when $k$ is a field), the Hochschild complexes
\[
   \CCC(\cA,h_{\cA^{\op}} \ten h_{\cA})  \quad   N\CCC(\cA,h_{\cA^{\op}} \ten h_{\cA})  
\]
associated to the Yoneda embedding $h_{\cA^{\op}} \ten h_{\cA}  \co \cA^{\op}\ten \cA\to \C_{\dg}(\cA^{\op}\ten \cA)$ are universal bialgebras, commutative whenever $\boxtimes$ is symmetric.

The coalgebra structure is given in Example \ref{HHunivcoalg}, and the multiplication and unit are given by the formulae of Proposition \ref{bialg}.
\end{example}

\begin{lemma}\label{Calglemma}
 Given a universal bialgebra $D$ and a strong monoidal dg functor $\omega$, the dg coalgebra $C:= \omega^{\vee}\ten_{\cA}D\ten_{\cA}\omega$ becomes a unital associative dg bialgebra, which is commutative whenever $D$ is commutative and $\omega$  symmetric.
\end{lemma}
\begin{proof}
 Since $\omega, \omega^{\vee}$ are strong monoidal dg functors, we have an isomorphism
\[
 \omega^{\vee}\ten_{\cA}(D\boxtimes^2 D)\ten_{\cA}\omega \cong (\omega^{\vee}\ten_{\cA}D\ten_{\cA}\omega)\ten_k  (\omega^{\vee}\ten_{\cA}D\ten_{\cA}\omega),
\]
so the multiplication $D\boxtimes^2 D \to D$ gives $C\ten_k C \to C$. Likewise, 
\[
 \omega^{\vee}\ten_{\cA}({\mathbbm{1}}\ten {\mathbbm{1}}) \ten_{\cA}\omega= \omega^{\vee}({\mathbbm{1}})\ten_k \omega({\mathbbm{1}}) \cong k,
\]
so the unit gives $k \to C$. Compatibility of the algebra and coalgebra structures follows from the corresponding results for $D$.
\end{proof}

\begin{remark}\label{delignermk}
 When $\cA$ is a neutral Tannakian category, taking duals gives an equivalence $\cA^{\op} \simeq \cA$. Then $\id_{\cA}\in \C(\cA^{\op}\ten \cA)$ corresponds to the ring of functions on Deligne's fundamental groupoid $G(\cA) \in \C(\cA\ten \cA)^{\op}$ from \cite[6.13]{droite}. Since $\id_{\cA}= \H^0(D)$, we thus think of $D$ as being the ring of functions on the path space of $\cA$.
\end{remark}

\subsubsection{Tilting modules}\label{monoidaltilt}

\begin{lemma}\label{Palglemma}
Given a universal bialgebra $D$ and a strong monoidal dg functor $\omega$, the tilting module $P:= D\ten_{\cA}\omega$  becomes a monoid in $\cD_{\dg}(\cA^{\op})$ with respect to $\boxtimes$, which is commutative whenever $D$ is commutative and $\omega$  symmetric.

Moreover, the co-action $P \to P\ten_k C$ of   \S \ref{tilt} is an algebra morphism in the sense that the diagram
\[
 \xymatrix{
  P\boxtimes P \ar[r] \ar[d] & (P\ten_kC) \boxtimes (P\ten_kC) \ar@{=}[r] & (P\boxtimes P)\ten_k(C\ten_kC)\ar[d]\\
P \ar[rr] && (P\ten_kC)    
 }     
\]
commutes, where the horizontal maps are co-action and the vertical maps are multiplication.
\end{lemma}
\begin{proof}
Since $\omega$ is a  strong monoidal dg functor, we have an isomorphism
\[
 (D\boxtimes^2 D)\ten_{\cA}\omega \cong (D\ten_{\cA}\omega)\boxtimes  (D\ten_{\cA}\omega),
\]
which gives the required multiplication, with existence of the unit coming from the isomorphism $({\mathbbm{1}}\ten {\mathbbm{1}})\ten_{\cA}\omega \cong {\mathbbm{1}}\ten_k \omega({\mathbbm{1}})\cong {\mathbbm{1}}$.

The final statement follows from compatibility of the algebra and coalgebra structures for $D$.
\end{proof}

\section{Comodules}\label{comodsn}

From now on, $k$ will be a field.
Throughout this section, we will fix a small $k$-linear dg category $\cA$, a $k$-linear dg  functor $\omega \co \cA \to \per_{\dg}(k)$, and a universal coalgebra $D \in \cD_{\dg}(\cA^{\op}\ten \cA)$ in the sense of \S \ref{univcoalg}. We write $C:= \omega^{\vee}\ten_{\cA}D\ten_{\cA}\omega$ and $P:= D\ten_{\cA}\omega$ for the associated dg coalgebra and tilting module.

\subsection{The Quillen adjunction}

\subsubsection{Model structure on dg comodules}

\begin{definition}
 Let $\C_{\dg}(C)$ be the dg category of right $C$-comodules in cochain complexes over $k$. Write $\C(C)$ for the underlying category $\z^0\C_{\dg}(C)$ of right $C$-comodules in cochain complexes, and $\cD(\C)$ for the homotopy category given by formally inverting quasi-isomorphisms.
\end{definition}

\begin{proposition}\label{modelcomod}
 There is a closed model structure on $\C(C)$ in which weak equivalences are quasi-isomorphisms and cofibrations are injections. Fibrations are surjections with kernel $K$ such that
\begin{enumerate}
 \item the graded module $K^{\#}$ underlying $K$ is injective as a comodule over the graded coalgebra $C^{\#}$ underlying $C$, and 
\item for all acyclic $N$, $\HHom_C(N,K)$ is acyclic.
\end{enumerate}
\end{proposition}
\begin{proof}
This is described in   \cite[Remark 8.2]{positselskiDerivedCategories}, as the model structure ``of the first kind''. For ease of reference, we summarise the arguments here.

As in \cite[Theorem 8.1]{positselskiDerivedCategories}, the lifting properties follow from the statement that $\HHom_C(E,I) \simeq 0$ whenever $I$ is fibrant and either $E$ or $I$ is acyclic. For $E$, this is tautologous. For $I$, note that the identity morphism in $\HHom_C(I,I)$ is then a coboundary, so we have a contracting homotopy $h$ with $[d,h]=\id$, implying that $\HHom_C(E,I) \simeq 0$ for all $E$. 

To establish factorisation, we first observe that we can embed any comodule $M$ into a quasi-isomorphic $C^{\sharp}$-injective comodule using a bar resolution
\[
 \bigoplus_{n\ge 0} M\ten C^{\ten n+1}[-n].
\]
Fibrant replacement then follows from a triangulated argument, \cite[Lemma 1.3]{positselskiDerivedCategories}. The key step is given in \cite[Lemma 5.5]{positselskiDerivedCategories}, where Brown representability gives a right adjoint to the functor from the coderived category to the derived category.
\end{proof}

We then write $\cD_{\dg}(C)$ for  the full dg subcategory of $\C_{\dg}(C)$ on fibrant objects.

\begin{remark}\label{generalbase1}
We might sometimes wish to consider multiple dg fibre functors. Given a set $\{\omega_x\}_{x \in X}$ of dg fibre functors, we can consider the coalgebroid $C$ on objects $X$ given by $C(x,y) := \omega_x^{\vee}\ten_{\cA}D\ten_{\cA}\omega_y$, with comultiplication $C(x,y) \to C(x,z)\ten_kC(z,y)$ and counit $C(x,x) \to k$ defined by the usual formulae. 

There is also a category $\C(C)$ of right $C$-comodules in cochain complexes, with such a  comodule $M$ consisting of cochain complexes $M(x)$ for each $x \in X$, together with a distributive action $M(y) \to M(x)\ten C(x,y)$. The proof of Proposition \ref{modelcomod} then adapts to give a closed model structure on the category $\C(C)$, noting that bar resolutions 
\[
  y \mapsto \bigoplus_{x \in X^{n+1}} M(x_0)\ten C(x_0, x_1) \ten \ldots \ten C(x_{n-1}, x_n)\ten C(x_n,y)[-n]
\]
still exist in this setting.
\end{remark}

\begin{definition}
 Given a left $C$-comodule $M$ and a right $C$-comodule $N$, set the cotensor product $N\ten^CM$ to be kernel of the map
\[
(\mu_N\ten_k\id_M - \id_N\ten_k\mu_M)\co  N\ten_kM \to N\ten_kC\ten_kM,       
\]
where $\mu$ denotes the $C$-coaction. Note that this is denoted by $N\boxempty_CM$ in \cite[2.1]{positselskiDerivedCategories}.
\end{definition}

\subsubsection{The Quillen adjunction}

\begin{lemma}\label{Qlemma}
The adjunction
\[
\xymatrix@1{ \C(\cA)  \ar@<1ex>[rr]^{ -\ten_{\cA}P}  && \C(C) \ar@<1ex>[ll]^{\HHom_C(P,-) }_{\bot}      }
\]
is a Quillen adjunction.
 \end{lemma}
\begin{proof}
It suffices to show that $-\ten_{\cA}P$ sends (trivial) generating  cofibrations to (trivial) cofibrations. Generating cofibrations are of the form $X\ten_k U \to X\ten_kV$ for $X \in \cA$ and $U \into V$ finite-dimensional cochain complexes. Now, $(X\ten_kU)\ten_{\cA}P= U\ten_kP(X)$, and $\ten_kP(X)$ preserves both injections and quasi-isomorphisms.
\end{proof}

\begin{definition}
Denote the co-unit of the Quillen adjunction by 
\[
 \vareps_N \co \HHom_C(P,N)\ten_{\cA}P \to N.
\]
\end{definition}

\subsubsection{The retraction}
From now on, we assume that our chosen $\ten_{\cA}$-coalgebra $D$ is ind-compact. 

\begin{proposition}\label{retractprop}
The counit 
\[
 \vareps_N \co \oR\HHom_C(P,N)\ten_{\cA}P\to N       
\]
of the derived adjunction $ (-\ten_{\cA}P)\dashv \oR\HHom_C(P,-)$  is an isomorphism in the derived category $\cD(C)$ for all $N$. 
 \end{proposition}
\begin{proof}
For any $C$-comodule $N$, we have the following isomorphisms
\begin{align*}
 \HHom_C(P,N)\ten_{\cA}P&= \HHom_C(P,N)\ten_{\cA}(\LLim_iP_i)\\
&\cong \LLim_i(\HHom_C(P,N)\ten_{\cA}P_i)\\
&\cong \LLim_i\HHom_{\cA}(P_i', \HHom_C(P,N))\\
&\cong \LLim_i\HHom_C(P_i'\ten_{\cA}P,N),
\end{align*}
where $P_i'$ is the predual of $P_i$, as in Definition \ref{predualdef}.

The co-unit $\vareps_N$ induces $C$-comodule morphisms
\[
\HHom_C(P_i'\ten_{\cA}P,N)= \HHom_C(P,N)\ten_{\cA}P_i \to N
\]
for all $i$, and hence $C_i$-comodule morphisms
\[
 \HHom_C(P_i'\ten_{\cA}P,N)\to N\ten^CC_i. 
\]

Now, since $C_i$ is finite-dimensional, we have $ N\ten^CC_i\cong \HHom_C(C_i^{\vee},N)$, so $\vareps$ induces $C$-bicomodule morphisms
\[
\alpha_i\co C_i^{\vee}\to P_i'\ten_{\cA}P,
\]
compatible with the transition maps $C_i \to C_j$, $P_i \to P_j$. 

As $P_i'$ is cofibrant, the quasi-isomorphism $P \to \omega$ induces a quasi-isomorphism $\beta_i \co P_i'\ten_{\cA}P\to P_i'\ten_{\cA}\omega= C_i^{\vee}$ of cochain complexes. Now, $\alpha_i$ is equivalent to the coaction map $P_i \to P_i\ten_kC_i\to P\ten_kC_i$, so $\beta_i \circ \alpha_i$ is equivalent to the coaction map $P_i \to \omega\ten_kC_i$. This is equivalent to the isomorphism $P_i'\ten_{\cA}\omega= C_i^{\vee}$, so $\beta_i \circ \alpha_i$ is the identity.

Therefore the  $\alpha_i$ are all quasi-isomorphisms, so for $N$ fibrant, the map
\[
 \HHom_C(P,N)\ten_{\cA}P_i \to N\ten^CC_i
\]
is a quasi-isomorphism. Since filtered colimits commute with finite limits, this gives a quasi-isomorphism
\begin{align*}
 \HHom_C(P,N)\ten_{\cA}P &= \LLim_i\HHom_C(P,N)\ten_{\cA}P_i\\
&\to \LLim_i N\ten^CC_i\\
&\cong N\ten^C(\LLim_iC_i )\\
&= N\ten^CC=N.
\end{align*}
\end{proof}


\subsection{Tannakian comparison}

We now show how for our chosen ind-compact universal coalgebra $D \in \cD_{\dg}(\cA^{\op}\ten \cA)$ and dg functor $\omega \co \cA \to \per_{\dg}(k)$, 
the tilting module $P= D\ten_{\cA}\omega$ can give rise to a comparison between the derived category of $\cA$-modules and the derived category of comodules of $C=\omega^{\vee}\ten_{\cA}D\ten_{\cA} \omega $. This is analogous to derived Morita theory (comparing two derived categories of modules) or Morita--Takeuchi theory (comparing two derived categories of comodules).


\begin{definition}
 Write $\ker \omega$ for the full dg subcategory of  $\cD_{\dg}(\cA)$ consisting of objects $X$ with $\omega(X):=X\ten_{\cA}\omega$ quasi-isomorphic to $0$.   

Recall from \cite[\S 12.6]{drinfeldDGQuotient}  that the right orthogonal complement $(\ker\omega)^{\perp}\subset  \cD_{\dg}(\cA)$ is the full dg subcategory consisting  of those $X$ for which $\HHom_{\cA^{\op}}(M,X) \simeq 0$ for all $M \in \ker \omega$.
\end{definition}

The following theorem contrasts strongly with \cite{ChuangLazarevMannan}, which shows that if we work with coderived categories of comodules instead of derived categories, and take a specific model for $C$, then we would not need to take the dg quotient by $\ker \omega$. As we will see, there are many applications in which the quotient is more interesting than the original dg category.

\begin{theorem}\label{tannakathm}
For the constructions of $C \simeq \omega^{\vee}\ten_{\cA}^{\oL}\omega$ and the tilting module $P$  above,
the derived adjunction $(-\ten_{\cA}P) \dashv \oR\HHom_C(P,-)$ gives rise to a quasi-equivalence between the dg categories $(\ker\omega)^{\perp}$ and  $\cD_{\dg}(C)$. Moreover, the map $(\ker\omega)^{\perp}\to \cD_{\dg}(\cA)/(\ker \omega)$  to the dg quotient is a quasi-equivalence.   
\end{theorem}
\begin{proof}
Functorial cofibrant and fibrant replacement give us composite dg functors $U\co \cD_{\dg}(C) \xra{\HHom_C(P,-) } \C_{\dg}(\cA) \to \cD_{\dg}(\cA)$ and $F\co \cD_{\dg}(\cA) \xra{\ten_{\cA}P} \C_{\dg}(C) \to \cD_{\dg}(C)$, and these will yield the quasi-equivalence.  

First observe that for $K \in \ker \omega$, we have  quasi-isomorphisms
\[
 K\ten_{\cA}P  \simeq \omega(K) \simeq 0      
\]
of cochain complexes, since $P$ is a resolution of $\omega$ and $K$ is cofibrant.

For any $N \in \cD_{\dg}(C)$, Proposition \ref{retractprop} gives that the counit 
\[
 \vareps_N \co \HHom_C(P,N)\ten_{\cA}P\to N       
\]
of the adjunction $ (-\ten_{\cA}P)\dashv \HHom_C(P,-)$  is a quasi-isomorphism. Thus for any $K \in \ker \omega$, we have
\[
 \oR\HHom_{\cA}(K, \oR\HHom_C(P,N))\simeq \oR\HHom_C(K\ten_{\cA}P,N) \simeq \oR\HHom_C(0,N)\simeq 0, 
\]
so $UN$ (the cofibrant replacement of   $\oR\HHom_C(P,N)$) lies in $(\ker\omega)^{\perp}$. 

Thus $F$ provides a retraction of $(\ker \omega)^{\perp}$ onto $\cD_{\dg}(C)$, and in particular $U\co \cD_{\dg}(C) \to (\ker \omega)^{\perp}$ is a full and faithful dg functor.

For any $M \in \cD_{\dg}(\cA)$, we now consider the unit
\[
 \eta_M \co M \to \oR\HHom_C(P,M\ten_{\cA}P)= \HHom_C(P, FM)
\]
of the adjunction. On applying $\ten_{\cA}P$, this becomes a quasi-isomorphism, with quasi-inverse $\vareps_{M\ten_{\cA}P}$, so $\omega\ten_{\cA}^{\oL}(\eta_M)$ is a quasi-isomorphism. Since $M$ is cofibrant, the map $\eta_M$ lifts to a map 
\[
 \tilde{\eta}_M \co M \to UFM
\]
of cofibrant objects, with
 \[
 \cone(\tilde{\eta}_M) \in \ker \omega, \quad \text{ i.e. } F\cone(\tilde{\eta}_M)\simeq 0.
\]
 
The dg subcategory $\ker \omega$ is thus right admissible in the sense of  
\cite[\S 12.6]{drinfeldDGQuotient}, because we have the morphism
\[
 M \xra{\tilde{\eta}_M} UFM
\]
for all $M \in \cD_{\dg}(\cA)$, with $UFM \in (\ker\omega)^{\perp} $ and $\cocone(\tilde{\eta}_M)\in  \ker \omega$.

In particular, this implies that if $M \in (\ker\omega)^{\perp}$, the map $\tilde{\eta}_M \co M \to UFM$ is a quasi-isomorphism, so $U\co \cD_{\dg}(C) \to (\ker \omega)^{\perp}$ is essentially surjective and hence a quasi-equivalence.

As observed in \cite[\S 12.6]{drinfeldDGQuotient}, the results of  \cite[\S 1]{BondalKapranovRepSerre} and \cite[\S I.2.6]{verdierCatDer} show that
right admissibility is equivalent to saying that $(\ker\omega)^{\perp}\to \cD_{\dg}(\cA)/(\ker \omega) $ is an equivalence.
\end{proof}

\begin{remark}\label{uniquenessrmk}
Note that Theorem \ref{tannakathm} implies that for any choices $D,D'$ of ind-compact $\ten_{\cA}$-coalgebra resolution of $\id_{\cA}$, the associated coalgebras  $C,C'$ are derived Morita equivalent. Given a quasi-isomorphism $D\to D'$, we then have a derived Morita equivalence $C \to C'$, which is \emph{a fortiori} a quasi-isomorphism.  

It might therefore seem curious that   $D \to \id_{\cA}$ is only required to be a quasi-isomorphism. However, any quasi-isomorphism to the trivial coalgebra $\id_{\cA}$ is automatically a Morita equivalence. The reason for this is that fibrant replacement in the category of $D$-comodules is given by the coaction $M \to M\ten D$, so the forgetful dg functor from $D$-comodules to $\id_{\cA}$-comodules is a quasi-equivalence.
\end{remark}

\begin{remark}\label{cfayoub0}
In \cite{ayoubGaloisMotivic1}, Ayoub establishes a weak Tannaka duality result for any monoidal functor $f\co \cM \to \cE$ of monoidal  categories equipped with a (non-monoidal) right adjoint $g$. He sets $H:= fg(\mathbbm{1})$,  shows (Theorem 1.21) that  $H$ has the natural structure of a biunital bialgebra, and then  (Propositions 1.28 and 1.55) proves that $f$ factors through the category of $H$-comodules, and that $H$ is universal with this property.

We may compare this with our setting by taking $\cM= \cD(\cA)$ and $\cE= \cD(k)$, the derived categories of $\cA$ and $k$. In this case, Ayoub's formula for the coalgebra underlying $H$ is defined provided $\cA$ and $f$ are $k$-linear, without requiring that $\cD(\cA)$ be monoidal. 

We can take $f$ to be $\ten_{\cA}^{\oL}\omega$, which has right adjoint $\oR\HHom_{k}(\omega,-)$ (with the same reasoning as Lemma \ref{Qlemma}). Thus 
\[
 H\simeq \oR\HHom_{k}(\omega,k) \ten_{\cA}^{\oL}\omega= \omega^{\vee}\ten_{\cA}^{\oL}\omega,
\]
which is the image  $[C] \in \cD(k)$ of our dg coalgebra $C \in \C_{\dg}(k)$.

One reason our duality results in Theorem \ref{tannakathm}  give a comparison rather than just universality is that we use the dg category of $C$-comodules in $\C_{\dg}(k)$. Instead, \cite[Proposition 1.55]{ayoubGaloisMotivic1} just looks at $C$-comodules in the derived category $\cD(k)$ --- in other words, (weak) homotopy comodules without higher coherence data. Likewise, his bialgebra $H$ is only defined as a (weak) homotopy bialgebra.

To recover Ayoub's weak universality from  Theorem \ref{tannakathm} in this setting, first observe that there is a forgetful functor from $\cD(C)$ to the category $\Co\Mod([C])$ of $[C]$-comodules in $\cD(k)$. The equivalence $\cD(\cA)/(\ker \omega)\simeq \cD(C)$ then ensures that $\omega \co \cD(\cA) \to \cD(k)$ factors through   $\Co\Mod([C])$. Likewise, if $\omega \co \per(\cA) \to \cD(k)$ factored through $\Co\Mod(B)$ for some other coalgebra $B \in  \cD(k)$,  Theorem \ref{tannakathm} would give an exact functor $\cD(C) \to \Co\Mod(B)$ fibred over $\cD(k)$. The image of $C$ would be a $B$-comodule structure on $[C] \in \cD(k)$,  compatible with the coalgebra structure of $[C]$ via the image of the comultiplication $C\ten C \to C$, giving a morphism  $[C] \to B$ of coalgebras in $\cD(k)$.

In \cite[Theorem 4.14]{iwanariTannakization}, Iwanari effectively gives a refinement of Ayoub's Tannaka duality. Starting from a monoidal $\infty$-functor of stable $\infty$-categories (a generalisation of dg categories), he constructs a derived affine group scheme $G$ (thus incorporating higher coherence data), and shows that its $\infty$-category $\Rep(G)$ of representations has a universal property, but without  
the characterisation $\Rep(G) \simeq (\ker\omega)^{\perp}$ of Theorem \ref{tannakathm} above --- this characterisation will be essential in our comparisons of categories of motives in Remark \ref{motayoub1}, and in Example
\ref{mothomb} and \cite[Example \ref{HHtannaka2-mothom2}]{HHtannaka2}.
\end{remark}

\subsection{Tannakian equivalence}

When $\omega$ is faithful, we now have statements about the 
 idempotent-complete pre-triangulated envelope  $\per_{\dg}(\cA)$ of  $\cA$, and its closure $\cD_{\dg}(\cA)=\ind(\per_{\dg}(\cA) )$  under filtered colimits:
\begin{corollary}\label{tannakacor}
Assume that  $\omega \co \cD_{\dg}(\cA) \to \C_{\dg}(k)$ is faithful in the sense that $\ker \omega$ is the category of acyclic $\cA$-modules, and take the tilting module $P$ and dg coalgebra $C \simeq \omega^{\vee}\ten_{\cA}^{\oL}\omega$ as above.
Then  the derived adjunction $(-\ten_{\cA}P) \dashv \oR\HHom_C(P,-)$ of Theorem \ref{tannakathm} gives rise to a quasi-equivalence between the dg categories $\cD_{\dg}(\cA)$ and  $\cD_{\dg}(C)$.

Moreover, the idempotent-complete pre-triangulated category $\per_{\dg}(\cA)$ generated by $\cA$ is quasi-equivalent  to the full dg subcategory of $\cD_{\dg}(C)_{\cpt}\subset  \cD_{\dg}(C)$ on objects which are compact in $\cD(C)$. If $\cA$ is Morita fibrant, this gives a quasi-equivalence $\cA \simeq \cD_{\dg}(C)_{\cpt}$.
\end{corollary}
\begin{proof}
 Since  $\ker \omega$ consists only of acyclic modules, $(\ker\omega)^{\perp}= \cD_{\dg}(\cA)$, and we apply Theorem \ref{tannakathm}. For the second part, note that 
the quasi-equivalence $-\ten_{\cA}P\co \cD_{\dg}(\cA)\to \cD_{\dg}(C)$ preserves filtered colimits, so $-\ten_{\cA}P\co \cD(\cA)\to \cD(C) $ preserves and reflects compact objects. Since $\H^0\per_{\dg}(\cA)\subset \cD(\cA)$ is the full subcategory on compact objects, the same must be true of its image in $\cD(C)$. Finally, if $\cA$ is Morita fibrant, then $\cA \to \per_{\dg}(\cA)$ is a quasi-equivalence.
\end{proof}

\begin{example}\label{kellerex}
Let $A=k[\eps]$ with $\eps^2=0$ (the dual numbers), and let $\omega$ be the $A$-module $k=k[\eps]/\eps$. This is faithful because for any cofibrant complex $M$ of $A$-modules, we have $\omega(M)= M/\eps M$, and a short exact sequence
\[
 0 \to \omega(M) \xra{\eps} M \to \omega(M) \to 0.
\]

A model for the cofibrant dg $\ten_A$-coalgebra $D$ in $A$-bimodules is given by 
\[
 D^{-n}= A\ten (k\xi_n)\ten A
\]
for $n \ge 0$,
with comultiplication given by 
\[
 \Delta(a\ten \xi_n\ten b)= \sum_{i +j=n} a\ten \xi_i \ten 1 \ten \xi_j \ten b \in A\ten (k \xi_i)\ten A \ten (k \xi_j) \ten A, 
\]
and counit $a\ten \xi_0 \ten b \mapsto ab \in A$. The differential is determined by
\[
  d(1 \ten \xi_1\ten 1)=(\eps \ten \xi_0 \ten 1) -(1\ten \xi_0 \ten \eps).      
\]

We therefore get $C:= k\ten_AD\ten_Ak = k\<\xi\>$, the free dg coalgebra on generator $\xi=\xi_1$ in degree $-1$, with $d\xi=0$. The tilting module $P$ is given by  $A\<\xi\>$, with left multiplication by $A$, right comultiplication by $C$, and $d\xi= \eps$. 

Thus the dg category of cofibrant $A$-modules is equivalent to the dg category of fibrant $k\<\xi\>$-comodules. Contrast this with  \cite[Example 2.5]{kellerHCexact}, which uses the tilting module $P$ to give a derived Morita equivalence between the category of all finitely generated $A$-modules and the dg category of perfect $k\<\xi\>^{\vee}$-modules.
\end{example}

\begin{remark}\label{generalbase}
 If we have a finite set $\{\omega_x\co \cA \to \per_{\dg}(k)\}_{x \in X}$ of dg fibre functors, we can form a dg coalgebroid on objects $X$ by $C(x,y)= \omega_x\ten_{\cA}D\ten_{\cA}\omega_y$, and then Theorem \ref{tannakathm}  adapts to give an equivalence between dg $C$-comodules and $(\bigcap_{x \in X} \ker \omega_x)^{\perp}$, using Remark \ref{generalbase1}. When the $\omega_x$ are jointly faithful, Corollary \ref{tannakacor} will thus adapt to give a quasi-equivalence between $\cD_{\dg}(\cA)$ and $\cD_{\dg}(C)$.

 Beware that if we had infinitely many dg fibre functors, the proof of Theorem \ref{tannakathm} would no longer adapt, because the expression $N\ten^CC_i$ in the proof of Proposition \ref{retractprop} would then be an infinite limit.
 
This also raises the question of a generalisation to faithful dg fibre functors $\omega \co \cA \to \C$ to more general categories. The obvious level of generality would replace $\per_{\dg}(k)$ with some rigid tensor category $\C$ over $k$.
In order to proceed further, we would need an extension of Theorem \ref{tannakathm} to deal with $\C$-coalgebras. In particular, generalisations would be required of the relevant model structures on comodules in \cite[8.2]{positselskiDerivedCategories}.
\end{remark}

\subsection{Homotopy invariance}\label{funsn}

\begin{corollary}\label{quasiequivcor}
 Given a $k$-linear dg functor  $\omega \co \cA \to \per_{\dg}(k)$ and a $k$-linear quasi-equivalence $F \co \cB \to \cA$, the morphism
\[
   C_{(\omega \circ F)}(\cB) \to C_{\omega}(\cA)
\]
of dg coalgebras induced by $F$ is a derived Morita equivalence, so \emph{a fortiori} a quasi-isomorphism.
\end{corollary}
\begin{proof}
By Theorem \ref{tannakathm}, the dg functor $  \cD_{\dg}(C_{(\omega \circ F)}(\cB))\to\cD_{\dg}(C_{\omega}(\cA))$ is quasi-equivalent to  $ \cD_{\dg}(\cB)/(\ker (\omega\circ F)) \to \cD_{\dg}(\cA)/(\ker \omega)$, which is a quasi-equivalence because $F$ is so. Therefore we have a derived Morita equivalence of dg coalgebras. 
\end{proof}

\begin{corollary}\label{QIMcor}
Given a natural quasi-isomorphism $\eta \co \omega \to \omega'$ of $k$-linear dg functors  $\omega, \omega' \co \cA \to \per_{\dg}(k)$, there is a span of morphisms between $C_{\omega}(\cA)$ and $C_{\omega'}(\cA)$ which are  derived Morita equivalences.  
\end{corollary}
\begin{proof}
If we let $I$ be the category with objects $0,1$ and a unique non-identity morphism $\pd \co 0 \to 1$, then $\eta$ defines a $k$-linear dg functor
\[
 \uline{\eta} \co \cA \by I \to \per_{\dg}(k)
\]
determined by $\uline{\eta}|_{\cA \by 0}= \omega$, $\uline{\eta}|_{\cA \by 1} = \omega'$ and $\uline{\eta}(\pd \co X \by 0 \to X \by 1)= \eta_X \co \omega(X) \to \omega'(X)$.    

Lemma \ref{coalgfunlemma} combined with the functors $0,1\to I$ then gives us morphisms
\[
 C_{\omega}(\cA) \to  C_{\uline{\eta}}(\cA\by I) \la C_{\omega'}(\cA)
\]
of dg coalgebras; we need to show these   are derived Morita equivalences.

By Theorem \ref{tannakathm}, this is equivalent to showing that the functors
\[
 \cD_{\dg}(\cA)/(\ker \omega) \xra{0^*} \cD_{\dg}(\cA \by I)/(\ker \uline{\eta}) \xla{1^*} \cD_{\dg}(\cA)/(\ker \omega')
\]
 are quasi-equivalences. 

For all $X \in \cA$, the cone $c_X$ of $h_{(X,0)} \to h_{(X,1)}$ lies in $\ker \uline{\eta}$ because $\eta$ is a quasi-isomorphism. For $M \in (c_X)^{\perp}$, this implies that the maps $M(X,1) \to M(X,0)$ are all quasi-isomorphisms. For $M,N \in  \cD_{\dg}(\cA \by I)$, the complex $\HHom_{\cA \by I}(M,N)$ is quasi-isomorphic to the cocone of 
\[
\HHom_{\cA}(M(0),N(0))\by \HHom_{\cA}(M(1),N(1))\to  \HHom_{\cA}(M(1),N(0)), 
\]
so the dg functors
\[
0^*, 1^* \co \cD_{\dg}(\cA) \to \cD_{\dg}(\cA \by I)/(\{c_X~:~ X \in \cA\})  
\]
are quasi-equivalences, as is their retraction $0_*$ given by  $M \mapsto M(0)$.  

We now just observe (by checking on representables $h_{(X,i)}$)  that there is a natural quasi-isomorphism from $\uline{\eta}$ to the dg functor $M \mapsto \cone(\omega M(1) \to \omega M(0) \oplus \omega'M(1))$. Since $\omega \to \omega'$ is a quasi-isomorphism, this is quasi-isomorphic to the dg functor $\omega\circ 0_*$, so  $\ker \uline{\eta}= \ker(\omega\circ 0_*) $.
\end{proof}

 In particular, this implies that the choice in Remark \ref{hfdrmk} does not affect the output, and that the constructions of \S \ref{laxsn} associate strong homotopy monoids to quasi-strong monoidal functors.

\subsection{Example: motives}\label{mothomsn}


Our main motivating example comes from the derived category of motives. 

As explained in \cite[\S 3]{ayoubRealEtOps}, there is a projective model structure on the category  $\cM$ of symmetric $T$-spectra in presheaves of $k$-linear complexes on the category $\Sm/S$ of smooth $S$-schemes. By \cite[Definitions 4.3.6 et 4.5.18]{ayoub6OpsII}, this has a left Bousfield localisation $\cM_{\bA^1}$, the projective $(\bA^1, \et)$-local model structure, whose homotopy category is  Voevodsky's triangulated category of motives over $S$ whenever $S$ is normal. These model categories are defined in terms of cochain complexes, so have the natural structure of dg model categories. 

Write $\cM_{\dg},\cM_{\dg,\bA^1}$ for the full dg subcategories on fibrant cofibrant objects --- this ensures that $\Ho(\cM) \simeq \H^0\cM_{\dg}$ and similarly for $\cM_{\bA^1}$. Take $ \cM_{\dg,c},\cM_{\dg,\bA^1,c}$ to be the full subcategories of  $\cM_{\dg},\cM_{\dg,\bA^1}$  on homotopically compact objects. For similar constructions along these lines, see \cite{BeilinsonVologodsky}. 

 By \cite[Proposition 2.1.6]{cisinskidegliseMixedWeil}, any choice of $k$-linear stable cohomology theory over $S$ gives a commutative ring object $\sE$ in  $\cM$, with the property that 
\[
 E:=\HHom_{\cM}(-,\sE) \co \cM_{\dg}^{\op} \to  \cD_{\dg}(k)
\]
 represents the cohomology theory. In particular, the functor is stable and $\bA^1$-invariant, so 
by \cite[Theorem 1]{cisinskidegliseMixedWeil}, $E$  gives a dg functor  from $\cM_{\dg,\bA^1,c}^{\op}$ to cohomologically finite complexes. 

Remark \ref{hfdrmk} allows us to replace this with a dg functor
\[
 \tilde{E} \co \tilde{\cM}_{\dg,\bA^1,c}^{\op} \to \per_{\dg}(k)       
\]
for some cofibrant replacement $\tilde{\cM}_{\dg,\bA^1,c}$ of $\cM_{\dg,\bA^1,c}$, and we can then form the dg coalgebra $C:=C_{\tilde{E}^{\vee}}(\tilde{\cM}_{\dg,\bA^1,c})$. By Corollary \ref{QIMcor}, this construction is essentially independent of the choice $\tilde{E} $ of replacement.

Theorem \ref{tannakathm} then gives a quasi-equivalence 
\[
 \cD_{\dg}(C) \simeq \cM_{\dg,\bA^1}/(\ker E)
\]
 between
the dg category of $C$-comodules  and the dg enhancement of the triangulated category of  motives modulo homologically acyclic motives. 
If $E'$ is the composition of $E$ with the derived localisation dg functor $ \tilde{\cM}_{\dg,c} \to\tilde{\cM}_{\dg,\bA^1,c} $, and $C':=C_{\tilde{E'}^{\vee}}(\tilde{\cM}_{\dg,c}) $ this also gives 
\[
 \cD_{\dg}(C) \simeq \cM_{\dg}/(\ker E') \simeq \cD_{\dg}(C'). 
\]

One consequence of the existence of a motivic $t$-structure over $S$ would be that $\cM_{\dg,\bA^1,c}$ lies in the right orthogonal complement $(\ker E)^{\perp}$, 
 in which case $ \cM_{\dg,\bA^1,c} $ would be quasi-equivalent to a full dg subcategory of $\cD_{\dg}(C)$.

These constructions  can all be varied by replacing $\cM$ with  the category $\cM^{\eff}$ of  presheaves of $k$-linear complexes on $\Sm/S$ with its projective model structure. This has a left Bousfield localisation $\cM_{\bA^1}$, by \cite[Definition  4.4.33]{ayoub6OpsII}, and 
 the homotopy category of $\cM_{\bA^1}^{\eff}$ is Voevodsky's triangulated category of effective motives when $S$ is normal, as in  \cite[Appendix B]{ayoubRealEtOps}.   There are natural dg enhancements, and we denote their restrictions to fibrant cofibrant objects by  $\cM^{\eff}_{\dg},\cM^{\eff}_{\dg,\bA^1}  $. 

Composing $E$ with the suspension functor $\Sigma^{\infty} \co \cM^{\eff}_{\dg, \bA^1} \to \cM_{\dg, \bA^1}$ defines a dg functor $E \co (\cM_{\dg, \bA^1}^{\eff})^{\op} \to  \cD_{\dg}(k)$, which is simply given by $\HHom_{\cM^{\eff}}(-,\sE(0))$.
We now get a dg coalgebra $C^{\eff}:=C_{\tilde{E}^{\vee}}(\tilde{\cM}_{\dg,\bA^1,c}^{\eff})$ with
\[
 \cD_{\dg}(C^{\eff}) \simeq \cM_{\dg,\bA^1}^{\eff}/(\ker E) \simeq \cM_{\dg}^{\eff}/(\ker E').
\]

A set of compact generators of $\cM_{\dg}^{\eff}$ is given by the set of presheaves $k(X)$ for smooth $S$-varieties $X$. If $\cA$ is the full subcategory on these generators and $E'$ has finite-dimensional values on $\cA$, then  we get a Morita equivalence between $C$  and $C_{E'}(\cA)$. Note that the latter is just given by the total  complex of
\[
n \mapsto \bigoplus_{X_0, \ldots X_n} E'(X_0) \ten k (X_0(X_1)\by X_1(X_2) \by \ldots \by X_{n-1}(X_n)) \ten E'(X_n)^{\vee},
\]
where we write $X(Y)= \Hom_S(Y,X)$, with the dg coalgebra structure of  Proposition \ref{coalg}. 

\begin{example}[Ayoub and Nori's motivic Galois groups]\label{motayoub1}
We now compare this with Ayoub's construction of a motivic Galois group from \cite[\S 2]{ayoubGaloisMotivic1}. For a number field $F$, he applies his Tannaka duality construction to the Betti realisation functor
\[
 \Ho(E)^{\vee} \co \H^0\cM_{\dg,\bA^1}(F,\Q) \to \cD(\Q)
\]
associated to an embedding $\sigma \co F \to \Cx$,
giving a Hopf algebra $\cH_{\mot}(F,\sigma) \in \cD(\Q)$. Replacing $\cM_{\dg,\bA^1}(F,\Q)$ with   
$\cM_{\dg,\bA^1}^{\eff}(F,\Q)$ gives a bialgebra $\cH_{\mot}^{\eff}(F,\sigma)\in \cD(\Q)$. 

From  Remark \ref{cfayoub0}, it follows that $ \cH_{\mot}(F,\sigma)$ and $\cH_{\mot}^{\eff}(F,\sigma)$ are  just the homotopy classes of our dg coalgebras $C, C^{\eff}$ above, equipped with their natural multiplications (and  in the former case, antipode) in the homotopy category coming from the rigid monoidal structure of $\Ho(E)$.

A variant of the construction above is given by considering generators of $\cM_{\dg}^{\eff}(\Q)$ given by $\cone(\Q(Y) \to \Q(X))[i]$, for Nori's good pairs $(X,Y,i)$ as in \cite[Definition  1.1]{HuberMuellerStach}. These have the property that their Betti realisations are cohomologically concentrated in degree $0$. Writing $\cA_{\Nori}$ for the full dg subcategory on these generators, we have $C^{\eff} \simeq C_{E'}(\cA_{\Nori})$. 

We also have $ C_{E'}(\z^0\cA_{\Nori}) \simeq C_{\H^0E'}(\z^0\cA_{\Nori})$ and $  \H^0C_{\H^0E'}(\z^0\cA_{\Nori}) \cong \H^0C_{\H^0E'}(\H^0\cA_{\Nori})$, and (by \cite[Theorem 7.3]{JoyalStreet}) comodules of the latter are precisely  Nori's abelian category $\cM\cM_{\Nori}^{\eff}$ of effective mixed motives as in \cite[Definition 1.3]{HuberMuellerStach}, since  the diagram $D^{\eff}$ of good pairs generates $\z^0\cA$. Then $ \Spec \H^0C_{E'}(\z^0\cA_{\Nori})$ is a pro-algebraic monoid whose group of units is Nori's Galois group $G_{\Nori}$. The inclusion $\H^0C_{E'}(\z^0\cA_{\Nori})  \into \H^0C^{\eff} $ thus induces a surjection $\Spec \H^0C \onto G_{\Nori}$. 

Contrast this with \cite[Remark 5.20]{iwanariTannakization}, highlighting that a comparison between Nori's and Voevodsky's motives is beyond the reach of Iwanari's Tannakian formulation. 
\end{example}

We now introduce alternative simplifications of the dg coalgebra in special cases.
\begin{remark}\label{bondarkormk} 
When $S$ is the spectrum of a field, the comparison of \cite[Appendix B]{ayoubRealEtOps} combines with the results of \cite{VSF} 
or \cite[Corollary 4.4.3]{cisinskideglise} 
to show that a set of generators of $ \cM_{\dg,\bA^1}$ is given by the motives of the form  $M_k(X)(r)$ for $X$ smooth and projective over $S$, and $r \in \Z$. Thus  the set of  motives of the form $M_{k,r}(X):=M_k(X)(r)[2r]$ is another generating set. For $X$ of dimension $d$ over $S$,  the dual of $M_{k,r}(X)$ is $M_{k,d-r}(X)$, so this set of generators is closed under duals, and
\[
 \cM_{\dg,\bA^1}( M_{k,r}(X) ,  M_{k,s}(Y)) \simeq\cM_{\dg,\bA^1} ( M_k(S),M_{k,d+s-r}(X\by_SY)).
\]

When $S$ is the spectrum of a perfect field, 
this implies that 
\[
 \H^i\cM_{\bA^1}( M_{k,r}(X) ,  M_{k,s}(Y)) \cong \CH^{ d+s-r}(X\by_SY, -i)\ten_{\Z}k,
\]
so these generators have no positive $\Ext$ groups between them, and we can replace the full dg category on these generators with its good truncation $\cB$ in non-positive degrees, given by 
\[
 \cB(  M_{k,r}(X) ,  M_{k,s}(Y)):= \tau^{\le 0}\cM_{\dg,\bA^1}(  M_{k,r}(X) ,  M_{k,s}(Y));
\]
we then have $\cD_{\dg}(\cB) \simeq \cM_{\dg,\bA^1}$.

The mixed Weil cohomology theory $E$ when restricted to $\cB$ thus admits a good truncation filtration, whose associated graded is quasi-isomorphic to 
\[
 \H^*E \co \H^0\cB^{\op} \to \C_{\dg}(k);
\]
think of this as a formal Weil cohomology theory. Note that this is a strong monoidal functor determined by the Chern character $\CH^{s}(Y) \to \H^{2s}(Y, E(s))$.

Since $\H^*E$ is finite-dimensional, we can then form the dg coalgebra $C_{\H^*E}(\cB)$, without needing to take a cofibrant replacement of $\cB$. 
Explicitly, this is given by the total  complex of
\begin{align*}
n \mapsto  \bigoplus_{X_0, \ldots X_n, r_0, \ldots, r_n} &\H^{*+2r_0}(X_0,E(r_0)) \ten_k \cB(M_{k,r_0}(X_0), M_{k,r_1}(X_1))  \ten_k\ldots \\
&\ldots \ten_k  \cB(M_{k,r_{n-1}}(X_{n-1}), M_{k,r_n}(X_n) ) \ten_k \H^{*+2r_n}(X_n,E(r_n))^{\vee},
\end{align*}
 with  the dg coalgebra structure of  Proposition \ref{coalg}.  We then have
\[
 \cD_{\dg}(C_{\H^*E}(\cB)) \simeq \cD_{\dg}(\cB)/\ker \H^*E \simeq \cD_{\dg}(\cB)/\ker E \simeq \cM_{\dg,\bA^1}/\ker E.
\]
\end{remark}

\begin{remark}\label{hanamurarmk}
 If we write $\cZ(X, \bt) $ for the $k$-linearisation of Bloch's cycle complex as in \cite[Proposition 1.3]{blochHigherK}, then we can follow \cite{hanamuraMMotAlgI} in defining $\cZ(X \by Y, \bt) \hten \cZ(Y \by Z, \bt) \subset \cZ(X \by Y, \bt) \ten_k \cZ(Y \by Z, \bt)$ to be the quasi-isomorphic subcomplex of cycles intersecting transversely. We then have a bicomplex 
\begin{align*}
 n \mapsto  \bigoplus_{X_0, \ldots X_n, r_0, \ldots, r_n} &\H^{*+2r_0}(X_0,E(r_0)) \ten_k \cZ^{d_0+ r_1-r_0}(X_0 \by X_1, -*)  \hten\ldots \\
&\ldots \hten  \cZ^{d_{n-1} +r_n-r_{n-1}}(X_{n-1}\by X_n, -*) \ten_k \H^{*+2r_n}(X_n,E(r_n))^{\vee},
\end{align*}
for $X_i$ of dimension $d_i$, and with differentials as in Definition \ref{Comegadef}.

The formulae of Proposition \ref{coalg}  make this into a dg coalgebra, which should be Morita equivalent to the dg coalgebra $C_{\H^*E}(\cB)$ of Remark \ref{bondarkormk}.
\end{remark}

\subsection{Monoidal comparisons}

We now consider the case where $(\cA,\boxtimes)$ is a monoidal dg category  and $\omega\co \cA \to \per_{\dg}(k)$ a strong monoidal dg functor.  We also assume that $D$ is a universal bialgebra in the sense of \S \ref{univbialg}.

Note that since $C$ is  a dg bialgebra by Lemma \ref{Calglemma}, the dg category $\C_{\dg}(C)$ has a monoidal structure $\ten_k$, where the coaction on $N\ten N'$ is the composition
\[
 N\ten N'\to (N\ten C) \ten (N'\ten C) \cong (N\ten N')\ten (C\ten C) \to (N\ten N')\ten C       
\]
of the co-actions with the multiplication on $C$.

\begin{lemma}\label{swapbox}
 For $M,M' \in \cD_{\dg}(\cA)$ and $N,N' \in \cD_{\dg}(\cA^{\op})$, there is a natural transformation
\[
 (M\ten_{\cA}N) \ten_k(M'\ten_{\cA}N')\to (M\boxtimes M')\ten_{\cA}(N\boxtimes N').
\]
\end{lemma}
\begin{proof}
 When $M,M'= h_X,h_Y$ and $N,N'= h_Y,h_Y'$ for $X,X',Y,Y'\in \cA$, this is just the map
\[
 \cA(X,Y)\ten_k \cA(X',Y') \to \cA(X\boxtimes X', Y\boxtimes Y')
\]
given by the bilinearity of $\boxtimes$. This extends uniquely to complexes and direct summands.
\end{proof}

\begin{proposition}\label{monoidalprop}
The dg functor $(-\ten_{\cA}P)\co \cD_{\dg}(\cA)\to \C_{\dg}(C)$ is lax monoidal, with the transformations
\[
  (M\ten_{\cA}P) \ten_k(M'\ten_{\cA}P)\to  (M\boxtimes M')\ten_{\cA}P      
\]
 being quasi-isomorphisms.
\end{proposition}
\begin{proof}
Lemma \ref{swapbox} gives the required transformations        
\[
 (M\ten_{\cA}P) \ten_k(M'\ten_{\cA}P)\to (M\boxtimes M')\ten_{\cA}(P\boxtimes P)\to (M\boxtimes M')\ten_{\cA}P.
\]
The quasi-isomorphism $P \to \omega$ then  maps these transformations quasi-isomorphically to 
\[
 (M\ten_{\cA}\omega) \ten_k(M'\ten_{\cA}\omega)\to (M\boxtimes M')\ten_{\cA}\omega,
\]
i.e. 
\[
 \omega(M)\ten_k \omega(M) \to \omega(M\boxtimes M'),
\]
which is an isomorphism because $\omega$ is required to be a strong monoidal dg functor.
\end{proof}

\begin{example}[Motives]\label{mothomb}
The model category $\cM^{\eff}$ of $k$-linear presheaves from \S \ref{mothomsn} is monoidal, as is its localisation $\cM_{\bA^1}^{\eff}$. However, the tensor product does not preserve fibrant objects, so the dg categories $ \cM_{\dg}^{\eff},\cM_{\dg,\bA^1}^{\eff}$ of fibrant cofibrant objects are not monoidal. [At best, they are multicategories (a.k.a. coloured operads), with $\HHom_{\cM_{\dg}^{\eff}}(X_1, \ldots, X_n;Y):= \HHom_{\cM^{\eff}}(X_1\ten \ldots \ten X_n;Y)$.]

However, if we take the dg category $\cM_{\dg}^{\eff'}$ of cofibrant objects in $\cM$, with $\cM_{\dg,c}^{\eff'}$ the full subcategory of   compact objects, then  $\cM^{\eff'}_{\dg}$ and $\cM^{\eff'}_{\dg,c} $ are monoidal dg categories. The dg categories $\cM_{\dg}^{\eff},\cM_{\dg,c}^{\eff}$  are quasi-equivalent to the dg quotients of $\cM^{\eff'}_{\dg}, \cM^{\eff'}_{\dg,c}$ by the class of weak equivalences in $\cM^{\eff}$. By \cite[Theorem 1]{cisinskidegliseMixedWeil}, a mixed Weil cohomology theory then gives rise to a contravariant monoidal dg functor $E$  from $\cM^{\eff'}_{\dg}$  to cohomologically finite complexes, by setting $E:=\HHom_{\cM^{\eff}}(-,\cE)$ for the associated presheaf $\cE$ of DGAs.

Since symmetric monoidal dg categories do not form a model category, we cannot then mimic the construction of \S \ref{mothomsn} and replace $E$ with a monoidal dg functor from a cofibrant replacement of $\cM^{\eff'}_{\dg} $ to finite-dimensional complexes. However, we can apply Proposition \ref{monoidalprop} if we can find a Weil cohomology theory taking values in finite-dimensional  complexes.

Objects of $\cM^{\eff'}_{\dg}$ are formal $k$-linear complexes of smooth varieties over $S$. When $S$ is a field admitting resolution of singularities, we can instead consider the model category $\cN^{\eff}$ of presheaves on the category of pairs $j \co U \to X$, where $X$ is smooth and projective over $S$, with $U$ the complement of a normal crossings divisor. Then for any mixed Weil cohomology theory $E$, there is an associated formal theory
\[
 E_f(j \co U \to X):= (\bigoplus_{a,b} \H^a(X, \oR^b j_* \cE_U, d_2),
\]
where $d_2$ is the differential on the second page of the Leray spectral sequence. Alternatively, this can be rewritten (as in \cite[3.2.4]{Hodge2}) in terms of $ \H^a(\tilde{D}^{b}, E(-b))$ and Gysin maps, where $\tilde{D}^{n} $ consists of local disjoint unions of  $n$-fold intersections in $D$. 

The constructions of \S \ref{mothomsn} all adapt from $\cM$ to $\cN$, and the restriction of $E_f$ to $\cN^{\eff'}_{\dg,c}$ takes values in finite-dimensional complexes, so we have a  dg bialgebra $C:=C_{E_f}(\cN^{\eff'}_{\dg,c} )$, and  Proposition \ref{monoidalprop} gives a monoidal dg functor
\[
 \cN^{\eff'}_{\dg} \to \C_{\dg}( C),
\]
inducing an equivalence $\cN^{\eff'}_{\dg}/\ker E \simeq  \cD_{\dg}( C)$. With some work (see Appendix \ref{formalweilapp}) 
it follows that $\cN^{\eff'}_{\dg}/\ker E \simeq \cM^{\eff'}_{\dg}/\ker E$ for all known Weil cohomology theories, and (for Betti cohomology) that  $E_f$ is quasi-isomorphic to $E$. 

Corollaries \ref{quasiequivcor} and \ref{QIMcor} then ensure that $C$ is essentially equivalent to the dg coalgebra of \S \ref{mothomsn}, so the constructions above give a strong compatibility result for the comparisons of \S \ref{mothomsn} with respect to  the monoidal structures. 

By Proposition \ref{bialg}, the multiplication on $C$ comes from the fibred dg functor $\boxtimes \co (\cN^{\eff'}\ten\cN^{\eff'}, \omega\ten \omega) \to (\cN^{\eff'}, \omega)$.  Applying    \cite[Corollary 1.14]{ayoubGaloisMotivic1} to these fibred dg categories gives bialgebra structures on the coalgebras  $[C\ten C], [C] \in \cD(k)$. By Ayoub's weak universal property \cite[Proposition 1.55]{ayoubGaloisMotivic1}, the functors $\boxtimes$ and $\mathbbm{1}$ on derived categories induce morphisms $[C\ten C] \to [C]$ and $[k] \to [C]$ of commutative bialgebras; the relations between these morphisms  force them to be Ayoub's multiplication and unit maps. Functoriality of the  comparison of universal constructions in Remark \ref{cfayoub0} then ensures that this weak bialgebra structure must come from our dg bialgebra structure.  
\end{example}

\appendix

\section{Formal Weil cohomology theories}

\subsection{Quasi-projective pairs and localisation}


\subsubsection{Quasi-projective pairs}

\begin{definition}
 Given a field $F$ admitting resolution of singularities, we let $\SmQP/F$ be the category of pairs $j \co U \to X$, where $X$ is smooth and projective over $F$, with $U$ the complement of a normal crossings divisor. We say that a morphism $(U,X) \to (U',X')$ in $\SmQP/F$ is an  equivalence (or in $\cE$) if it induces an isomorphism  $U\to U'$.
\end{definition}

\begin{lemma}\label{calcfrac}
 The pairs $(\SmQP/F, \cE)$ and $(\cE, \cE)$ admit right calculi of fractions in the sense of \cite[\S 7]{simploc2}.
\end{lemma}
\begin{proof}
We begin with the case  $(\SmQP/F, \cE)$, noting that $\cE$ contains all identities and is closed under composition. For any diagram $(V,Y) \xra{f} (U,X) \xla{a} (U,X')$ in $\SmQP/F $ (so  $a$ is an equivalence), we first need to find  a commutative diagram
\[
 \begin{CD}
  (V,Y') @>{f'}>> (U,X')\\
@V{b}VV @VV{a}V \\
(V,Y) @>{f}>> (U,X),
 \end{CD}
\]
with $b$ an equivalence. To do this, we first form the fibre product $X'\by_XY$, and observe that the isomorphism $V  =U\by_UV$ gives us a map $V \into X'\by_XY$. Taking $Y'$ to be a resolution of singularities of the closure of $V $ in $X'\by_XY$ gives the required diagram.

Secondly, we need to show that if any parallel arrows $f, g \co (V,Y) \to (U,X')$ in $\SmQP/F $ satisfy $af=ag$ for some equivalence  $a\co (U, X') \to (U,X)$, then there exists an equivalence $b\co (V,Y') \to (V,Y)$ with  $fb=gb$.  The condition $af=ag$ implies that the maps $f,g\co V \to U$ are equal. There is therefore a diagonal map
\[
 V \into Y\by_{f,X',g}Y.
\]
Taking $Y'$ to be a resolution of singularities of the closure of $V$ in $Y\by_{f,X',g}Y$ then gives the construction required. Thus  $(\SmQP/F, \cE)$ admits a right calculus of fractions.

Finally, note that $\cE$ satisfies the two out of three property, so as observed in \cite[7.1]{simploc2}, it follows that $(\cE,\cE)$ admits a right calculus of fractions.
\end{proof}

\subsubsection{Localisation and DG quotients}

\begin{definition}
 Given a category $\C$ and a subcategory $\cW$, we follow \cite[\S 7]{simploc2} in writing $\C[\cW^{-1}]$ for the localised category given by formally  inverting all morphisms in $\cW$.
\end{definition}

\begin{definition}
 Given a category $\C$ and a subcategory $\cW$, and an object $Y \in \C$, we  write 
\[
 \C\cW^{-1}(X,Y)
\]
for the category whose objects are spans
\[
 Y \xla{u} Y' \xra{f} X
\]
with $u$ in $\cW$, and whose morphisms are commutative diagrams
\[
 \begin{CD}
  Y @<{u_1}<<  Y_1 @>f_1>> X\\
@| @VV{v}V  @| \\
 Y @<{u_2}<<  Y_2 @>f_2>> X,
 \end{CD}
\]
with $v$ in $\cW$.
\end{definition}
Note that this category is denoted in \cite[5.1]{simploc} by $ N^{-1}\C\cW^{-1}(Y,X)$. 

\begin{definition}
 Given a category $\C$, write $k\C$ for the $k$-linear category with the same objects as $\C$, but with morphisms given by the free $k$-modules
\[
 (k\C)(X,Y):= k(\C(X,Y)).
\]
\end{definition}

\begin{definition}
 Given $\sF \in \C_{\dg}(k\C)$ (i.e. a contravariant dg functor from $\C$ to cochain complexes over $k$)  and $Y \in \C$, define the cochain complex $\sF\cW^{-1}(Y)$ by
\[
 \sF\cW^{-1}(Y):= \holim_{\substack{\longrightarrow \\Y' \in\cW\da Y}} \sF(Y').
\]
Explicitly, this can be realised as the direct sum total complex of the simplicial cochain complex
\[
 \bigoplus_{Y_0' \to Y} \sF(Y_0') \Leftarrow \bigoplus_{Y_1' \to Y_0' \to Y} \sF(Y_0') \Lleftarrow \ldots .
\]
Beware that this construction is not functorial in $Y$.
\end{definition}

\begin{proposition}\label{locprop}
Take a small category $\C$ and a subcategory $\cW$ such that $(\C,\cW)$ and $(\cW,\cW)$  admit right calculi of fractions. Let $\cD$ be the localised category $\C[\cW^{-1}]$ given by formally  inverting all morphisms in $\cW$. Then the functor $\lambda\co \C \to \cD$ gives a left  Quillen functor
\[
\lambda_!\co \C_{\dg}(k\C) \to \C_{\dg}(k\cD),
\]
left adjoint to $\lambda^{-1}$, making $\C_{\dg}(k\cD)$ Quillen-equivalent to the left Bousfield localisation of $\C_{\dg}(k\C)$ at the image $k\cW$ of $\cW$ under the Yoneda embedding $k\co \C\to \C_{\dg}(k\C)$.
\end{proposition}
\begin{proof}
The functor $\lambda_!$ satisfies $\lambda_!(k C) = k \lambda(C)$, which then determines $\lambda_!$  by right Kan extension. We begin by computing  this for cofibrant $k\C$-modules. 

Combining \cite[Propositions 7.2 and 7.3]{simploc2}, the morphism
\[
 B\C\cW^{-1}(X,Y) \to \cD(\lambda X, \lambda Y)
\]
 is a weak equivalence of simplicial sets for all $X,Y \in \C$. Since $k\cD(\lambda X, \lambda Y) = (\lambda_! kX)(\lambda Y)$,  and $B\C\cW^{-1}(X,Y)  = (k X)\cW^{-1}(Y)$, this gives a quasi-isomorphism 
 \[
  (k X)\cW^{-1}(Y) \to (\lambda_! kX)(\lambda Y),
 \]
  functorial in $X$ (but not in $Y$).
Since any cofibrant $k\C$-module $\sF$ is a retraction of a filtered colimit of finite complexes of $kX$'s, this gives quasi-isomorphisms
\[
 \sF \cW^{-1}(Y) \xra{\sim} (\lambda_! \sF)(\lambda Y)
\]
for all $Y \in \C$.

Now, the unit $\sF \to \lambda^{-1}\lambda_!\sF$ of the adjunction gives maps
\[
 \sF(Y) \to (\lambda_!\sF)(\lambda Y)
\]
for all $Y \in \C$, and these factor through the maps above, giving
\[
  \sF(Y) \to \sF \cW^{-1}(Y) \xra{\sim} (\lambda_!\sF)(\lambda Y).
\]
The $k\C$-module  $\sF$ will be  $k\cW$-local if and only if $\sF$ maps morphisms in $\cW$ to quasi-isomorphisms. If this is the case, then the map $\sF(Y) \to \sF\cW^{-1}(Y)$ is a quasi-isomorphism, so the unit
\[
 \sF(Y) \to (\lambda_!\sF)(\lambda Y)
\]
is also a quasi-isomorphism.

 Because $\lambda$ is essentially surjective on objects, the functor $\lambda^{-1}$ reflects quasi-isomorphisms. Thus the co-unit $\oL \lambda_! \lambda^{-1}\sG \to \sG$ of the derived adjunction is a quasi-isomorphism for all $\sG$.
Since $\lambda$ maps $\cW$ to isomorphisms, any object in the image of $\lambda^{-1}$ is $k\cW$-local. It therefore suffices to show that for any cofibrant $\sF \in DG(k\C)$, the unit $\sF(Y) \to (\lambda_!\sF)(\lambda Y)  $  of the adjunction is a $k\cW$-local equivalence. 
Now, for any $k\cW$-local object $\sG$ 
\begin{eqnarray*}
\oR\HHom_{k\C}(\sF, \sG)  &\simeq& \oR\HHom_{k\C}(\sF, \lambda^{-1}\oL \lambda_!\sG)\\
  &\simeq& \oR\HHom_{k\cD}( \lambda_!\sF, \oL \lambda_!\sG)\\
&\simeq&  \oR\HHom_{k\cD}(\oL \lambda_!\lambda^{-1} \lambda_!\sF, \oL \lambda_!\sG)\\
&\simeq&  \oR\HHom_{k\C}(\lambda^{-1} \lambda_!\sF,\lambda^{-1} \oL \lambda_!\sG)\\
&\simeq&\oR\HHom_{k\C}(\lambda^{-1} \lambda_!\sF,\sG),
\end{eqnarray*}
as required.
\end{proof}

\begin{corollary}\label{loccor}
In the setting of Proposition \ref{locprop},
 the functor $\lambda_! $ gives  a quasi-equivalence $(k\cW)^{\perp} \to \cD_{\dg}(k\cD)$  of dg categories. Moreover, the map $(\cW)^{\perp}\to \cD_{\dg}(k\C)/\cD_{\dg}(k\cW)$  to the dg quotient is a quasi-equivalence.  
\end{corollary}
\begin{proof}
First observe that $(k\cW)^{\perp} \subset \cD_{\dg}(k\C) $ consists of the fibrant cofibrant objects in   the  Bousfield model structure, automatically giving the quasi-equivalence $(k\cW)^{\perp} \to \cD_{\dg}(k\cD)$. 

Fibrant replacement in the Bousfield model structure gives us morphisms $r \co M \to \hat{M}$ in  for each $M \in \cD_{\dg}(k\C)$,
with $\hat{M} \in (k\cW)^{\perp}= \cD_{\dg}(k\cW)^{\perp}$ and $\cone(r) \in \cD_{\dg}(k\cW)$. Thus $\cD_{\dg}(k\cW)$ is right admissible in the sense of \cite[\S 12.6]{drinfeldDGQuotient}, giving the quasi-equivalence $(\cW)^{\perp}\to \cD_{\dg}(k\C)/\cD_{\dg}(k\cW)$.
\end{proof}

Now write $\Sm/F$ for the category of smooth schemes over $F$.

\begin{corollary}\label{QPcor}
 The excision functor $(X,D) \mapsto X\backslash D$ induces  quasi-equivalences  $\cD_{\dg}(\SmQP/F,k)/ \cD_{\dg}(k\cE) \la (k\cE)^{\perp} \to \cD_{\dg}(\Sm/F,k)$.
\end{corollary}
\begin{proof}
 This comes from applying Lemma \ref{calcfrac} to Corollary \ref{loccor}, noting that the excision functor is essentially surjective, so gives an equivalence
\[
 (\SmQP/F)[\cE^{-1}] \simeq \Sm/F.
\]
\end{proof}

\subsubsection{Formal Weil cohomology theories}\label{formalweilapp}

As in Example \ref{mothomb}, for any $k$-linear mixed Weil cohomology theory $E$ over the field $F$, we can now define the formal Weil cohomology theory
\[
 E_f \co (\SmQP/F)^{\op} \to \C_{\dg}(k)
\]
by
\[
 E_f(U \xra{j} X):= (\bigoplus_{a,b} \H^a(X, \oR^b j_* E_U, d_2),
\]
where $d_2$ is the differential on the second page of the Leray spectral sequence.

Weight considerations or standard results on Gysin maps imply that
the Leray spectral sequence degenerates at $E_2$ (at least for all known Weil theories), so  any equivalence $(U,X) \to (U,X')$ in $\SmQP/F $ induces a quasi-isomorphism on $E_f$.

Writing $\cN^{\eff'}_{\dg}:= \cD_{\dg}(\SmQP/K,k)$ as in Example \ref{mothomb}, the functor $E_f$ extends $k$-linearly, giving
\[
 E_f^{\vee}\co \cN^{\eff'}_{\dg}/ \cD_{\dg}(k\cE) \to \C_{\dg}(k),
\]
 since $\cE$ lies in the kernel of $E_f$. By Corollary \ref{QPcor}, $ \cN^{\eff'}_{\dg}/ \cD_{\dg}(k\cE) $ is quasi-equivalent to $\cM_{\dg}^{\eff'}$.

Moreover, since the Leray spectral sequence degenerates, we have $\ker E_f = \ker E$ on $ \cN^{\eff'}_{\dg}$, so
\[
 \cN^{\eff'}_{\dg}/\ker E_f = \cN^{\eff'}_{\dg}/\ker E \simeq \cM_{\dg}^{\eff'}/\ker E.
\]

 \subsection{Mixed Hodge structures on Betti cohomology}

\begin{definition}
For $\Lambda \subset \R$ a subfield, define $\MHS_{\Lambda}$ to be the tensor category of mixed Hodge structures in finite-dimensional vector spaces over $\Lambda$. Explicitly, an object of $\MHS_{\Lambda}$ consists of a finite-dimensional vector space $V$ over $\Lambda$ equipped with an increasing (weight) filtration $W$, and a decreasing (Hodge) filtration $F$ on $V\ten_{\Lambda}\Cx$ (both exhaustive and Hausdorff), such that 
\[
 \gr_F^p\gr_{\bar{F}}^q\gr^W_nV=0
\]
for $p+q\ne n$.

The functor forgetting the filtrations is faithful, so by Tannakian duality there is a corresponding affine group scheme which (following \cite{arapurapi1}) we refer to as the universal Mumford--Tate group $\MT_{\Lambda}$; this allows us to identify $\MHS_{\Lambda}$ with the category of finite-dimensional $\MT_{\Lambda}$-representations.


Denote the pro-reductive quotient of $\MT_{
\Lambda}$ by $\PMT_{\Lambda}$ --- representations of this correspond to Hodge structures (i.e. direct sums of pure Hodge structures) over $\Lambda$. The assignment of weights to pure Hodge structures defines a homomorphism $\bG_{m,\Lambda} \to\PMT_{\Lambda}$. 
\end{definition}

\begin{definition}
 For $\Lambda \subset \R$ a subfield, 
a $\Lambda$-Hodge complex in the sense of \cite[Definition 3.2]{beilinson} is a tuple
$(V_F, V_{\Lambda},V_{\Cx},\phi, \psi)$, where $(V_{\Lambda},W)$ is a filtered complex of $\Lambda$-modules, $(V_{\Cx},W)$ is a filtered complex of complex vector spaces, $(V_F,W, F)$ a bifiltered complex of complex vector spaces, and 
\[
\phi\co V_{\Lambda}\ten_{\Lambda}\Cx \to V_{\Cx} \quad \psi\co V_F \to  V_{\Cx}
\]
are $W$-filtered quasi-isomorphisms; 
these must also satisfy the conditions that
\begin{enumerate}
 \item the cohomology $\bigoplus_i \H^i(V_{\Lambda})$ is finite-dimensional over $\Lambda$;

\item for any $n \in \Z$, the differential in the filtered complex $(\gr^W_nV_F, \gr^W_nF)$ is strictly compatible with the filtration, 
or equivalently the map $\H^*(F^p\gr^W_nV_F) \to \H^*(\gr^W_nV_F)$ is injective;

\item the induced Hodge filtration together with the isomorphism $\H^i(\gr^W_n V_{\Lambda})\ten_{\Lambda}\Cx \to \H^i(\gr^W_n V_F)$ defines a pure $\Lambda$-Hodge structure of weight $n$ on $\H^i(\gr^W_n V_{\Lambda})$.
\end{enumerate}
\end{definition}
 
\begin{example}\label{qprojhodgecx}
For a sheaf  $\sF$ on $Y(\Cx)$, write $\sC^{\bt}_Y(\sF)$ for the Godement resolution of $\sF$ --- this is a cosimplicial diagram  of flabby sheaves. Write $\CC^{\bt}(X,\sF)$ for the global sections of $\sC^{\bt}_X(\sF)$.

Take a smooth projective complex variety $X$ and a complement  $j\co Y \into X$ of a normal crossings divisor $D$.
Then  set $A_{\Lambda}^{\bt}:= \CC^{\bt}(X, j_*\sC^{\bt}_Y(\Lambda))$,    $A_{\Cx}^{\bt}:= \CC^{\bt}(X, j_*\sC^{\bt}_Y(N_c^{-1}\Omega_Y^{\bt}))$, and $A_F^{\bt}:= \CC^{\bt}(X, N_c^{-1}\Omega_X^{\bt}\<D\>)$, where $N_c^{-1}$ is the Dold--Kan denormalisation functor from cochain complexes to cosimplicial modules. The filtration $W$ is given by d\'ecalage of the good truncation filtration on $ j_*$ in each case. 
Then $N_cA$ is  mixed Hodge complex, and  on applying the Thom--Sullivan functor $\Th$ from cosimplicial DG algebras to DG algebra, we obtain a commutative algebra $\Th(A)$ in mixed Hodge complexes.   
\end{example}

\begin{definition}
For $\Lambda \subset \R$ a subfield, define $\MHS_{\Lambda}$ to be the category of mixed Hodge structures in finite-dimensional vector spaces over $\Lambda$, and write $\Pi(\MHS_{\Lambda})$ for the group scheme over $\Lambda$ corresponding to the forgetful functor from $\MHS_{\Lambda}$ to $\Lambda$-vector spaces. 

\end{definition}

\begin{definition}
 Given a cosimplicial vector space $V^{\bt}$ and a simplicial set $K$, define $(V^{\bt})^K$ to be the cosimplicial vector space given by $ ((V^{\bt})^K)^n= (V^n)^{K_n}$, with operations $\pd^i\co ((V^{\bt})^K)^n\to 
((V^{\bt})^K)^{n+1}$ defined by composing
\[
 (V^n)^{K_n} \xra{(\pd^i)^{K_n}} (V^{n+1})^{K_n}\xra{ (V^{n+1})^{\pd_i}}(V^{n+1})^{K_{n+1}}, 
\]
and operations $\sigma^i\co ((V^{\bt})^K)^n\to 
((V^{\bt})^K)^{n-1}$ defined  similarly.

In particular, $(V^{\bt})^{\Delta^1}$ is a path object over $V^{\bt}$, with the two vertices $ \Delta^0 \to \Delta^1$ inducing two maps $(V^{\bt})^{\Delta^1} \to V^{\bt}$.
 \end{definition}

\begin{definition}
The $\Lambda$-algebra $O(\MT_{\Lambda})$ admits both left and right multiplication by $\MT_{\Lambda}$. These induce two different ind-mixed Hodge structures on $O(\MT_{\Lambda})$, which we refer to as the left and right mixed Hodge structures  $(O(\MT_{\Lambda}),W^l, F_l),(O(\MT_{\Lambda}),W^r, F_r)$ .  
\end{definition}

\begin{example}\label{qprojMHS}
 Given $A^{\bt}_{\Lambda, \cH}(X,D):= ( A_{\Lambda}^{\bt}, \phi, A_{\Cx}^{\bt}, \psi, A_F^{\bt})$ as in Example \ref{qprojhodgecx}, we can  define $A^{\bt}_{\MHS}(X,D;\Lambda)$ to be the limit of the diagram
\[
\xymatrix@R=0ex{
 (W \ten W^l)_0 (A_{\Lambda}^{\bt}\ten_{\Lambda}O(\MT_{\Lambda})) \ar[r] & (W\ten W^l)_0(A_{\Cx}^{\bt}\ten_{\Lambda}O(\MT_{\Lambda})) \\
(W\ten W^l)_0((A_{\Cx}^{\bt})^{\Delta^1}\ten_{\Lambda}O(\MT_{\Lambda})) \ar[ur] \ar[r] & (W\ten W^l)_0(A_{\Cx}^{\bt}\ten_{\Lambda}O(\MT_{\Lambda}))\\
(W\ten W^l)_0(F\ten F_l)^0(A_F^{\bt}\ten_{\Lambda}O(\MT_{\Lambda})) \ar[ur],
}
\]
giving a cosimplicial algebra. The right Hodge structure on $O(\MT_{\Lambda})$ then gives us a cosimplicial algebra
\[
 (A^{\bt}_{\MHS}(X,D;\Lambda), W^r, F_r)
\]
in $\ind(\MHS_{\Lambda})$.
\end{example}

\begin{proposition}\label{qprojmhsworks}
 The $\Lambda$-Hodge complex associated to the cosimplicial algebra $A^{\bt}_{\MHS}(X,D;\Lambda)$ in $\ind(\MHS_{\Lambda})$ is canonically quasi-isomorphic to  $A^{\bt}_{\Lambda, \cH}(X,D)$ as a commutative algebra in cosimplicial $\Lambda$-Hodge complexes.
\end{proposition}
\begin{proof}
If we set 
\begin{align*}
 B_{\Lambda}&:= A_{\Lambda}^{\bt}\by_{A_{\Cx}^{\bt}}(A_{\Cx}^{\bt})^{\Delta^1}\\
B_{\Cx}&:=(A_{\Cx}^{\bt})^{\Delta^1}\\
B_F&:=(A_{\Cx}^{\bt})^{\Delta^1}\by_{A_{\Cx}^{\bt}}A_F^{\bt},
\end{align*}
then $B:= (B_{\Lambda}, B_{\Cx}, B_F)$ has the natural structure of a $\Lambda$-Hodge complex, and the morphism $\sigma^0 \co \Delta^1 \to \Delta^0$ induces a quasi-isomorphism $A^{\bt}_{\Lambda, \cH}(X,D) \to B$.

There is a map from the $\Lambda$-Hodge complex associated to $A^{\bt}_{\MHS}(X,D;\Lambda)$ to $B$ given by projections. We need to show that these projections preserve the Hodge and weight filtrations, and are (bi)filtered quasi-isomorphisms.

Now, for any mixed Hodge structure $V$ there is a canonical isomorphism $V\cong V\ten^{\MT_{\Lambda}}O(\MT_{\Lambda}):= (V\ten O(\MT_{\Lambda}))^{\MT_{\Lambda}}$, where the  Mumford--Tate action combines the action on $V$ with the left action on $O(\MT_{\Lambda})$. The mixed Hodge structure on $V$ then corresponds to the Mumford--Tate action on $(V\ten O(\MT_{\Lambda}))^{\MT_{\Lambda}}$ induced by the right action on $O(\MT_{\Lambda})$. 

Since $W_nV$ is a sub-MHS, it follows that $W_nV \cong V\ten^{\MT_{\Lambda}}W_n^rO(\MT_{\Lambda})$, and since $W_n$ is an idempotent functor, this is also isomorphic to $(W_nV) \ten^{\MT_{\Lambda}}W_n^rO(\MT_{\Lambda})$. In particular,
\[
 V\ten^{\MT_{\Lambda}}W_m^lW_n^rO(\MT_{\Lambda})\subset (W_nV) \ten^{\MT_{\Lambda}}W_m^lO(\MT_{\Lambda}) \subset \sum_{\substack{i \le n, j\le m,\\ i+j=0}} (W_iV) \ten (W_j^lO(\MT_{\Lambda})),  
\]
which is $0$ for $m+n<0$. Thus $W_{-n-1}^lW_n^rO(\MT_{\Lambda})=0$. A similar argument shows that $F^{1-p}_lF^p_r(\sO(\MT_{\Lambda})\ten_{\Lambda}\Cx)=0$.

Now, the weight filtration  $W_nA^{\bt}_{\MHS}(X,D;\Lambda)$ is given by replacing $O(\MT_{\Lambda})$ with $W_n^rO(\MT_{\Lambda})$ in the definition of $A^{\bt}_{\MHS}(X,D;\Lambda)$, and the Hodge filtration $F^pA^{\bt}_{\MHS}(X,D;\Lambda)$ by replacing $O(\MT_{\Lambda})\ten_{\Lambda}\Cx$ with $F^p_r(O(\MT_{\Lambda})\ten_{\Lambda}\Cx)$. Projection onto the first factor gives a map from $W_nA^{\bt}_{\MHS}(X,D;\Lambda)$ to $\sum_i  (W_iA_{\Lambda}^{\bt})\ten W_{-i}^lW^r_nO(\MT_{\Lambda})$, which by the vanishing above is contained in $(W_nA_{\Lambda}^{\bt})\ten O(\MT_{\Lambda})$. Using similar arguments for the other factors and composing with the co-unit $O(\MT_{\Lambda}) \to \Lambda$ gives compatible (bi)filtered morphisms
\begin{align*}
 A^{\bt}_{\MHS}(X,D;\Lambda)&\to B_{\Lambda}\\
 A^{\bt}_{\MHS}(X,D;\Lambda)\ten_{\Lambda}\Cx &\to B_{\Cx}\\
A^{\bt}_{\MHS}(X,D;\Lambda)\ten_{\Lambda}\Cx &\to B_F,
\end{align*}
and it only remains to establish quasi-isomorphism.

The data $N_cA^{\bt}_{\Lambda,\cH}(X,D)$  of Example \ref{qprojhodgecx} define a $\Lambda$-Hodge complex, so by \cite[Theorem 3.4]{beilinson}, there exists a complex $V^{\bt}$ of mixed Hodge structures whose associated Hodge complex is quasi-isomorphic to $N_cA^{\bt}_{\Lambda, \cH}(X,D)$.

Now, observe that $N_cA^{\bt}_{\MHS}(X,D;\Lambda)$ is a cocone calculating absolute Hodge cohomology, so 
\begin{align*}
 N_cA^{\bt}_{\MHS}(X,D;\Lambda) &\simeq \oR\Gamma_{\cH}(A^{\bt}_{\Lambda, \cH}(X,D)\ten (O(\MT_{\Lambda}), W^l, F_l) )\\
&\simeq \oR\Gamma_{\cH}(V^{\bt}\ten (O(\MT_{\Lambda}), W^l, F_l))\\
&\simeq \oR\HHom_{\MHS,\Lambda}(\Lambda, V^{\bt}\ten (O(\MT_{\Lambda}), W^l, F_l))\\
&\simeq \HHom_{\MHS,\Lambda}(\Lambda, V^{\bt}\ten (O(\MT_{\Lambda}), W^l, F_l))\\
&\cong V^{\bt}\ten^{\MT_{\Lambda}}O(\MT_{\Lambda})\\
&\cong V^{\bt},
\end{align*}
with the last two properties following because $V^{\bt}\ten O(\MT_{\Lambda})$ is an injective $\MT_{\Lambda}$-representation and because $\ind(\MHS_{\Lambda})$ is equivalent to the category of $O(\MT_{\Lambda})$-comodules in $\Lambda$-vector spaces. The quasi-isomorphisms above all respect mixed Hodge structures (via the right action on $O(\MT_{\Lambda})$), completing the proof.
\end{proof}

\subsection{Splittings for Betti cohomology}

Since $\MT_{\Lambda}$ is an affine group scheme, it is an inverse limit of linear algebraic groups, so  by \cite{Levi}, there exists a Levi decomposition $\MT_{\Lambda} \cong \PMT_{\Lambda}\ltimes \Ru(\MT_{\Lambda})$ of the universal Mumford--Tate group as the semidirect product of its pro-reductive quotient and its pro-unipotent radical. Beware that this decomposition is not canonical; it might be tempting to think that the functor $V \mapsto \gr^WV$ yields the required section by Tannaka duality, but it is not compatible with the fibre functors. 

Moreover, Levi decompositions are conjugate under the action of the radical $  \Ru(\MT_{\Lambda})$, so the set of decompositions is isomorphic to the quotient
\[
  \Ru(\MT_{\Lambda})/ \Ru(\MT_{\Lambda})^{\PMT_{\Lambda}}
\]
 by the centraliser of $ \PMT_{\Lambda}$. For any element $u$ of $\Ru(\MT_{\Lambda})$, we must have $(u-\id)W_nV \subset W_{n-1}V$ for all mixed Hodge structures $V$. However, any  element
in the centraliser necessarily has weight $0$ for the $\bG_m$-action, so must be $1$. Thus the set of Levi decompositions is a torsor under
\[
 \Ru(\MT_{\Lambda}).
\]

\begin{proposition}\label{bettiformalprop} 
 Each choice of Levi decomposition for the universal Mumford--Tate group $\MT_{\Lambda}$ gives rise to a zigzag of $W$-filtered  quasi-isomorphisms between the cosimplicial algebra-valued  functors
\begin{align*}
 (X,D) &\mapsto A^{\bt}_{\Lambda}(X,D)\\
(X,D) &\mapsto N_c^{-1}( \H^*(X, \oR^*j_*\Lambda), d_2), 
\end{align*}
where $d_2$ is the differential on the $E_2$ page of the Leray spectral sequence and $j \co X\backslash D \to X$.
\end{proposition}
\begin{proof}
 A choice of Levi decomposition is equivalent to a retraction of $\MHS_{\Lambda}$ onto $\HS_{\Lambda}$, and $V \in \MHS_{\Lambda}$ is canonically isomorphic to $\gr^WV$. Since the weight filtration is a functorial filtration by mixed Hodge substructures, it is necessarily preserved by any such retraction, which thus amounts to giving a functorial $W$-filtered isomorphism $V \cong \gr^WV$ for all mixed Hodge structures $V$.

Proposition \ref{qprojmhsworks} gives a zigzag of functorial $W$-filtered quasi-isomorphisms between the cosimplicial algebra  $A^{\bt}_{\MHS}(X,D;\Lambda)$ and the Betti complex $A^{\bt}_{\Lambda}(X,D)$. A choice of Levi decomposition then gives a $W$-filtered isomorphism $A^{\bt}_{\MHS}(X,D;\Lambda) \cong \gr^W A^{\bt}_{\MHS}(X,D;\Lambda)$. Applying Proposition \ref{qprojmhsworks} to the associated gradeds then gives a zigzag of filtered quasi-isomorphisms between $\gr^W A^{\bt}_{\MHS}(X,D;\Lambda) $ and $ \gr^WA^{\bt}_{\MHS}(X,D;\Lambda)$, which maps quasi-isomorphically to $N_c^{-1}( \H^*(X, \oR^*j_*\Lambda), d_2)$.
\end{proof}

\begin{corollary}\label{bettiformalcor}
If $E_B$ denotes the mixed Weil cohomology theory associated to Betti cohomology, and $E_{B,f}$ its formal analogue as in Examples \ref{mothomb} and \S \ref{formalweilapp}, then each choice of Levi decomposition for the universal Mumford--Tate group $\MT_{\Q}$ gives a zigzag of    quasi-isomorphisms between $E_B$ and $E_{B,f}$.
\end{corollary}
\begin{proof}
 The functor $E_B$ is given by $X \mapsto \Th(\CC^{\bt}(X(\Cx),\Q))$, so there is a canonical quasi-isomorphism from $\Th(A^{\bt}_{\Lambda}(X(\Cx),D(\Cx))) $ to $E_B(X)$.  Proposition \ref{bettiformalprop} thus gives a zigzag of quasi-isomorphisms from $E_B(X)$ to $\Th N_c^{-1}(\H^*(X(\Cx), \oR^*j_*\Lambda), d_2)$, functorial in $(X\backslash D \xra{j} X)$ in $\SmQP/F$. The functors $\Th$ and $N_c^{-1}$ are homotopy inverses, so this is quasi-isomorphic to $(\H^*(X(\Cx), \oR^*j_*\Lambda), d_2)$, which is just $E_{B,f}(X\backslash D \xra{j} X)$.
\end{proof}

\bibliographystyle{alphanum}
\bibliography{references.bib}
\end{document}